\documentclass[journal,twoside,web]{ieeecolor}
\usepackage{generic}
\usepackage{cite}
\usepackage{amsmath,amssymb,amsfonts}
\usepackage{algorithmic}
\usepackage{graphicx}
\usepackage{textcomp}
\usepackage{bm}
\usepackage{caption,subcaption}
\usepackage{comment}
\usepackage{xcolor}
\let\labelindent\relax
\usepackage{enumitem}
\newtheorem{theorem}{Theorem}   
\newtheorem{lemma}{Lemma}

\newtheorem{proposition}{Proposition}

\newtheorem{example}{Example}
\newtheorem{remark}{Remark}

\newtheorem{definition}{Definition}
\newtheorem{assumption}{Assumption}

\newcommand{\norm}[1]{\left\lVert#1\right\rVert}

\newcommand\inner[2]{\left\langle #1, #2 \right\rangle}

\newcommand{\calH}{\mathcal{H}}
\newcommand{\calU}{\mathcal{U}}
\newcommand{\calV}{\mathcal{V}}
\newcommand{\calW}{\mathcal{W}}
\newcommand{\calY}{\mathcal{Y}}
\newcommand{\calX}{\mathcal{X}}
\newcommand{\calZ}{\mathcal{Z}}
\newcommand{\calL}{\mathcal{L}}
\newcommand{\calB}{\mathcal{B}}

\newcommand{\inv}{^{-1}}
\newcommand{\sq}{^{\frac{1}{2}}}
\usepackage{optidef}
\def\BibTeX{{\rm B\kern-.05em{\sc i\kern-.025em b}\kern-.08em
    T\kern-.1667em\lower.7ex\hbox{E}\kern-.125emX}}
\begin{document}
\title{Kernel-based models for system analysis}
\author{Henk J. van Waarde and Rodolphe Sepulchre
\thanks{Henk van Waarde is with the Automatic Control Laboratory, ETH Z\"urich, Switzerland. Rodolphe Sepulchre is with the Control Group, University of Cambridge, UK. (hvanwaarde@ethz.ch, rs771@cam.ac.uk). \\
This work was completed while the first author was a postdoctoral research associate at Cambridge. The work was supported by the European Research Council under the Advanced ERC Grant Agreement Switchlet n. 670645.}}

\maketitle

\begin{abstract}
This paper introduces a computational framework to identify nonlinear input-output operators that fit a set of system trajectories while satisfying incremental integral quadratic constraints. The data fitting algorithm is thus regularized by suitable input-output properties required for system analysis  and control design. This biased identification problem is shown to admit the tractable solution of a regularized least squares problem when formulated in a suitable reproducing kernel Hilbert space. The kernel-based framework is a departure from the prevailing state-space framework. It is motivated by  fundamental limitations of nonlinear state-space models at combining the fitting requirements of data-based modeling with the input-output requirements of system analysis and physical modeling. 
\end{abstract}


\section{Introduction}
\label{s:introduction}

Machine learning  has dramatically improved our ability to
algorithmically train nonlinear models from input-output data. A key challenge
for a broader application of machine learning to system design  
is to make the learning task compatible with system analysis and control
design. It is  one thing to guarantee that a model regresses a given
data set of  input-output trajectories, but it is another thing to
ensure that plugging this model in a feedback loop, or more generally
in an interconnected network,  comes with robustness and performance
guarantees traditionally associated to control system design.
This challenge is not confined to control design. It has pervaded
much recent research in machine learning, with motivations 
including the training of recurrent neural networks \cite{Erichson2021}, that is, neural
networks that include feedback loops, the design of training algorithms
robust  to adversarial attacks \cite{Eykholt2018} and, more generally, 
the ubiquitous need of regularizing data fitting with a priori knowledge \cite{Scholkopf2001}.

The purpose of the present paper is to bridge the data-based learning problem
with the  input-output analysis of control theory. 
Input-output analysis is a classical topic. It was  pioneered in the 1960s by Sandberg  \cite{Sandberg1964}, Zames \cite{Zames1966,Zames1966b} and Willems \cite{Willems1969}.
It advocates to ground the system analysis of an interconnected system in the incremental input-output
properties of the subsystems. One of the building blocks of this theory is the small-gain theorem, that determines the incremental
input-output gain of a feedback interconnection from the incremental gains of each subsystem in the loop. 
The focus  on {\it incremental} system properties, e.g.  {\it incremental} input-output gains,  is central to the present paper. 
Incremental properties hold for arbitrary deviations from {\it arbitrary} trajectories, in contrast to non-incremental properties
that only hold for arbitrary deviations from a {\it specific} trajectory.  Only for linear systems do system properties imply incremental system properties.
The emphasis on incremental properties goes back to the PhD thesis of Zames [7], and the theory of  incremental input-output analysis
is covered in  the classical textbook of Desoer and Vidyasagar \cite{Desoer1975}. However, incremental input-output analysis has lost its importance in the subsequent decades, for reasons further discussed in the paper.

In today's language of system theory, incremental input-output properties are best characterized by {\it incremental}
integral quadratic constraints (iIQCs), which express the positivity of a suitable inner product between {\it differences} of
input and output trajectories. Those constraints have a close relationship with the monotonicity and nonexpansiveness properties
of operator theory \cite{Bauschke2011}, which have become central to analysis and design questions in convex optimization \cite{Ryu2022}. 
IQCs were introduced in the seminal paper \cite{Megretski1997}, where they were shown to provide a unifying and computationally tractable framework for system analysis.
To date, they have been mostly used to study non-incremental system properties. The paper \cite{Kulkarni2002} is an early caveat of the non-trivial gap that separates IQC analysis from  incremental IQC analysis. 

The objective of the present paper is to explore a potential bridge between data-based modeling and system analysis. 
We address the basic question of fitting a given data set of input-output trajectories with a model
that satisfies a given incremental integral quadratic constraint. Our main result is to demonstrate that this question has a tractable answer within
the theory of reproducing kernel Hilbert spaces \cite{Aronszajn1950,Scholkopf2001,Micchelli2005}.  Via a specific linear (scattering) transformation, the general problem is reformulated as the 
specific problem of learning a model with small incremental gain, which itself boils down to a standard regularized least squares problem within a suitable Hilbert space. 
The significance of the result is that it combines the computationally tractable framework of kernel-based methods with the versatility
and generality of system analysis via incremental integral quadratic constraints. 

Kernel-based modeling is a departure from the state-space modeling framework that currently
prevails in nonlinear system theory and in data-based learning. This departure is motivated
by important limitations of nonlinear state-space models to provide
dynamical models that encode suitable incremental input-output
properties. This limitation is not widely acknowledged, but it is
particularly explicit in circuit theory, where fundamental physical input-output properties such as incremental passivity
fail to be retained in the simplest state-space models of nonlinear circuit elements. The difficulty of applying system analysis to the biophysical nonlinear
circuit models of neuroscience is in fact a very motivation of the present paper. It will be further illustrated by examples in the paper.

The use of kernel-based models in system theory is not novel. The theory of reproducing kernel spaces has been primarily exploited in  system identification
 \cite{Pillonetto2010, Pillonetto2014}. Those results have led to a renaissance of the field 
\cite{Pillonetto2011,Dinuzzo2015,Bottegal2016,
Bouvrie2017,Bouvrie2017b,Ramaswamy2018,
Khosravi2019,Dallalibera2021} and are today considered as a paradigm shift \cite{Ljung2020}. 
But the use of kernel-based system identification in system analysis and control design
has so far remained limited because it faces the same challenge as machine learning: identified models can serve control design
only if they come with suitable input-output properties. The present paper provides an avenue to achieve this objective, indicating
which types of kernels are suitable for system analysis. Incidentally, the most successful stable spline kernels of system identification \cite{Pillonetto2014} 
do not meet the required assumptions of the present paper, motivating new developments to bridge kernel-based system identification
and control design, that is, to advance the important topic of {\it identification for control} \cite{Gevers1995,Lindqvist2001,Bombois2004,Gevers2005}.

Contraction theory, pioneered in
the seminal paper \cite{Lohmiller1998}, has been the paradigm of choice to study incremental properties of {\it closed} dynamical systems, described by
ordinary differential equations without inputs. While contraction theory has been very successful at solving an increasing set of
nonlinear control problems,  we currently lack an analog framework for input-output analysis. This is in sharp contrast with the theory of linear systems,
where dissipativity theory provides the required bridge between input-output properties  and their state-space realization. 
Again, it is the limited success of  {\it incremental} dissipativity theory in addressing that objective to date that motivates the kernel-based framework of this paper. 
Kernels are regarded as a surrogate of state-space models to provide a tractable computational framework for system analysis
of incremental properties.

The organization of this paper is as follows. In Section~\ref{s:preliminaries} we give an overview of the notation, and we provide preliminary definitions on integral quadratic constraints and dissipativity (Section~\ref{s:preliminariesdissipativity}). In Section~\ref{s:preliminariesscattering} we explain how incremental integral quadratic constraints for arbitrary quadratic supply rates can be reduced to incremental gain properties via \emph{scattering} \cite{vanderSchaft2017}. We also point out the relation to nonexpansive and monotone operators in Section~\ref{s:preliminariesmonotone}. 

Next, in Section~\ref{s:problem} we state the problem: to identify operators from data that satisfy causality properties and incremental integral quadratic constraints. Section~\ref{s:regularizedleastsquares} then discusses the solution concept, namely to solve a regularized least squares problem over a suitable reproducing kernel Hilbert space. 

The subsequent sections contain several results that show how particular choices of reproducing kernels and regularization parameters ensure that the least squares solution satisfies given system-theoretic properties. For example, Section~\ref{s:nonexpansive} focuses on the nonexpansiveness property. Section~\ref{s:nonexpansivetodissipative} shows how to recover operators satisfying incremental integral quadratic constraints from nonexpansive ones. Thereafter, Section~\ref{s:causalandincdis} discusses the selection of kernels for causal and incrementally dissipative operators. In Section~\ref{s:potassiumcurrent} we apply the results of previous sections to the identification of a monotone model of the potassium current. Also kernels for non-incremental system-theoretic properties are discussed in Section~\ref{s:otherproperties}. 
In Section~\ref{s:conclusions} we provide our conclusions. Finally, in the Appendix we give an extensive review of relevant results within the theory of reproducing kernel Hilbert spaces of operators. 

\section{Preliminaries}
\label{s:preliminaries}

Let $\calX$ and $\calZ$ be real Banach spaces with norms $\norm{\cdot}_\calX$ and $\norm{\cdot}_\calZ$, respectively. We denote the collection of all \emph{linear operators} from $\calX$ to $\calZ$ by $\calL(\calX,\calZ)$. The class of \emph{bounded} linear operators from $\calX$ to $\calZ$ is denoted by $\calB(\calX,\calZ)$, and consists of all $A \in \calL(\calX,\calZ)$ for which there exists a constant $c \in \mathbb{R}$ such that $\norm{A(x)}_\calZ \leq c \norm{x}_\calX$. We denote the operator norm of $A \in \calB(\calX,\calZ)$ by $\norm{A}_{\calB(\calX,\calZ)}$. If $\calX = \calZ$ we simply use the notation $\calL(\calX)$ and $\calB(\calX)$. The identity operator in $\calB(\calX)$ is denoted by $I$. 

Next, let $\calX, \calZ$ be real Hilbert spaces with inner product $\inner{\cdot}{\cdot}_\calX$ and $\inner{\cdot}{\cdot}_\calZ$. We use $A^*$ to denote the \emph{adjoint} of $A \in \calB(\calX,\calZ)$, i.e., $A^*$ is the unique operator in $\calB(\calZ,\calX)$ for which
$$
\inner{y}{A(x)}_\calZ = \inner{A^*(y)}{x}_\calX
$$
for all $x\in\calX$ and $y\in \calZ$. An operator $A\in\calB(\calX)$ is called \emph{self-adjoint} if $A = A^*$. It is called \emph{positive} if $\inner{x}{Ax}_\calX \geq 0$ for all $x \in \calX$. If $A \in \calB(\calX)$ is a self-adjoint positive operator then there exists a unique self-adjoint positive $A^{\frac{1}{2}} \in \calB(\calX)$, called the \emph{square root}, such that $A = A^{\frac{1}{2}}A^{\frac{1}{2}}$ \cite[p. 265]{Riesz1956}.

Let $\mathbb{R}^+$ denote the nonnegative reals and let $\calV$ be a finite dimensional inner product space with inner product $\inner{\cdot}{\cdot}_\calV$ and induced norm $\norm{\cdot}_\calV$. The set of \emph{square integrable functions} $L_2(\mathbb{R}^+,\calV)$ consists of all measurable $f: \mathbb{R}^+ \to \calV$ such that 
\begin{equation}
\label{L2definition}
\int_{0}^\infty \norm{f(t)}_\calV^2 dt < \infty.
\end{equation}
By identifying functions that are equal except for a set of Lebesgue measure zero, $L_2(\mathbb{R}^+,\calV)$ becomes a Hilbert space with inner product  
$$
\inner{f}{g}_{L_2(\mathbb{R}^+,\calV)} := \int_0^\infty \inner{f(t)}{g(t)}_\calV dt.
$$
The space $L_2([0,\tau],\calV)$ can be defined similarly, by replacing the improper integral \eqref{L2definition} by an integral from $0$ to $\tau$. The Hilbert space of \emph{square summable sequences} $\ell_2(\mathbb{N},\calV)$ consists of all sequences $s:\mathbb{N} \to \calV$ such that 
\begin{equation}
\label{definitionl2}
\sum_{k = 0}^\infty \norm{s(k)}_\calV^2 < \infty.
\end{equation}
The space $\ell_2(\{0,\dots,\tau\},\calV)$ can be defined similarly. Sometimes we will drop the arguments and simply write $L_2$ and $\ell_2$ when this does not lead to confusion.

\subsection{Integral quadratic constraints and dissipativity}
\label{s:preliminariesdissipativity}

We denote time by $\mathbb{T}$, which we assume to be either the half line $\mathbb{R}^+$ or $[0,\tau]$ for continuous-time, or the natural numbers $\mathbb{N}$ or $\{0,1,\dots,\tau\}$ in the discrete case. Compatible with $\mathbb{T}$, we let $\calU$ be either the space $L_2(\mathbb{T},\mathbb{R}^m)$ in continuous-time or $\ell_2(\mathbb{T},\mathbb{R}^m)$ in discrete-time. Similarly, $\calY$ is either $L_2(\mathbb{T},\mathbb{R}^p)$ or $\ell_2(\mathbb{T},\mathbb{R}^p)$.
For $f:\mathbb{T} \to \mathbb{R}^k$ we will use $P_T f : \mathbb{T} \to \mathbb{R}^k$ to denote the truncated signal 
\begin{equation}
\label{definitionPT}
P_T f(t) = \begin{cases} f(t) & \text{ if } 0 \leq t \leq T \\
0 & \text{ if } t > T. \end{cases} 
\end{equation} 
Note that since $\calU$ is either $L_2$ or $\ell_2$, we have that $P_T u \in \calU$ for all $T \in \mathbb{T}$ and $u \in \calU$. To define the relevant notions of dissipativity, we introduce the \emph{supply rate} $s_\Phi : \mathbb{R}^m \times \mathbb{R}^p \to \mathbb{R}$ as 
\begin{equation}
\label{supplyfunction}
s_\Phi(x_1,x_2) := \begin{bmatrix}
x_1 \\ x_2
\end{bmatrix}^\top \Phi \begin{bmatrix}
x_1 \\ x_2
\end{bmatrix},
\end{equation}
where $\Phi \in \mathbb{R}^{(m+p)\times(m+p)}$ is symmetric. 
\begin{definition}
\label{d:allsystemproperties}
Consider the supply rate $s_\Phi$ in \eqref{supplyfunction} and the operator $H:\calU\to\calY$. We say that $H$:
\begin{enumerate}[label=(\roman*)]
\item \label{i:incavnonneg} satisfies an \emph{incremental integral quadratic constraint} if\footnote{In discrete-time, all integrals of Definition~\ref{d:allsystemproperties} should be replaced by sums and the terminology ``sum quadratic constraint" would be more appropriate.}
\begin{equation}
\label{incrementalaveragenonnegativity}
\int_{\mathbb{T}} s_\Phi(u(t)-v(t),y(t)-z(t)) \: dt \geq 0,
\end{equation}
for all $u,v \in \calU$ and $(y,z) = (H(u),H(v))$.
\item \label{i:incdis} is \emph{incrementally dissipative} if
\begin{equation}
\label{incrementaldissipativity}
\int_0^T s_\Phi(u(t)-v(t),y(t)-z(t)) \: dt \geq 0
\end{equation} 
for all $T \in \mathbb{T}$, all $u,v\in\calU$ and $(y,z) = (H(u),H(v))$.
\item \label{i:causal} is \emph{causal} if 
\begin{equation}
\label{causal}
P_T(H(u)) = P_T(H(P_T u ))
\end{equation} 
for all $u \in \calU$ and $T \in \mathbb{T}$.
\item \label{i:avnonneg} satisfies an \emph{integral quadratic constraint} if
\begin{equation}
\label{averagenonnegativity}
\int_{\mathbb{T}} s_\Phi(u(t),y(t)) \: dt \geq 0,
\end{equation}
for all $u \in \calU$ and $y = H(u)$.
\item \label{i:dis} is \emph{dissipative} if
\begin{equation}
\label{dissipativity}
\int_0^T s_\Phi(u(t),y(t)) \: dt \geq 0
\end{equation}
for all $T \in \mathbb{T}$, all $u\in\calU$ and $y = H(u)$.
\end{enumerate}
\end{definition}

It is straightforward to prove that for causal operators $H$, the properties \ref{i:incavnonneg} and \ref{i:incdis}, and the properties \ref{i:avnonneg} and \ref{i:dis} are equivalent, see related arguments in \cite{Desoer1975,vanderSchaft2017}. In the special case that $m = p$ and
\begin{equation}
\label{supplypassivity}
\Phi = \begin{bmatrix}
0 & I \\ I & 0
\end{bmatrix},
\end{equation} 
(incremental) dissipativity boils down to the notion of \emph{(incremental) passivity}, see e.g. \cite[p. 184]{Desoer1975}. Moreover, in the special case that 
$$
\Phi = \begin{bmatrix}
\delta I & 0 \\ 0 & -I
\end{bmatrix}
$$
for $\delta > 0$, the properties \ref{i:incdis} and \ref{i:dis} retrieve the notions of \emph{finite (incremental) $L_2$-gain}, see \cite[Def. 2.1.5]{vanderSchaft2017}. We note that we consider $L_2$ or $\ell_2$ as input and output spaces instead of their extended spaces, as done in \cite{Desoer1975,vanderSchaft2017}. This has the disadvantage of allowing a somewhat smaller class of signals but, as we will see, it enables the application of the powerful theory of reproducing kernels (that requires Hilbert spaces). The terminology of integral quadratic constraint (IQC) originates from Megretski and Rantzer \cite{Megretski1997}, see also \cite{Veenman2013,Lessard2016}. We additionally point out that for particular choices of $\Phi$, \eqref{incrementalaveragenonnegativity} can be regarded as a \emph{monotonicity} or \emph{nonexpansiveness} property of the operator $H$ on the Hilbert space $\calU$; these concepts arise naturally in convex analysis and optimization \cite{Rockafellar2009,Bauschke2011,Ryu2022} and we will discuss this further in Section~\ref{s:preliminariesmonotone}.

\subsection{Scattering}
\label{s:preliminariesscattering}

We make the following blanket assumption on $\Phi$.

\begin{assumption}
\label{a:Phi}
The symmetric matrix $\Phi \in \mathbb{R}^{(m+p)\times(m+p)}$ is nonsingular, has $p$ negative eigenvalues and (thus) $m$ positive eigenvalues.
\end{assumption}

Note that both examples of $\Phi$ mentioned in Section~\ref{s:preliminariesdissipativity} satisfy Assumption~\ref{a:Phi}. This allows us to factor $\Phi$ as 
\begin{equation}
\label{factorPhi}
\Phi = M^\top \underbrace{\begin{bmatrix}
I & 0 \\ 0 & -I
\end{bmatrix}}_{=:\Sigma}
M,
\end{equation} 
where $M \in \mathbb{R}^{(m+p)\times(m+p)}$ is a nonsingular matrix. By Assumption~\ref{a:Phi}, the identity in $\Sigma$ is of dimension $m$, while the minus identity block has dimension $p$. Compatible with the partitioning of $\Sigma$, we partition $M$ as 
\begin{equation}
\label{matrixM}
M = \begin{bmatrix}
M_{11} & M_{12} \\ M_{21} & M_{22}
\end{bmatrix},
\end{equation}
where $M_{11}\in\mathbb{R}^{m\times m}$, $M_{12} \in \mathbb{R}^{m \times p}$, $M_{21} \in \mathbb{R}^{p \times m}$ and $M_{22} \in \mathbb{R}^{p \times p}$. With some abuse of terminology, we regard $M_{11}$ as an operator on $\calU$ by defining $M_{11}f \in \calU$ as the function $t \mapsto M_{11}f(t)$ for all $t \in \mathbb{T}$. Similar terminology is used for $M_{21}$, $M_{12}$ and $M_{22}$. 

\begin{proposition}
\label{p:scattering}
Consider $\Phi$ as in Assumption~\ref{a:Phi} and factored as in \eqref{factorPhi}. Let $R: \calU \to \calY$ be such that $M_{11}+M_{12}R$ is invertible. Define the operator $S: \calU \to \calY$ as
\begin{equation}
\label{operatorS}
S := (M_{21}+M_{22}R)(M_{11}+M_{12}R)\inv.
\end{equation}
Let $(u,y),(v,z) \in \calU \times \calY$ be related by
$(v,z) = M(u,y)$. 
Then we have that
\begin{enumerate}[label=(\roman*)]
\item \label{i:MN} $N_{11}+N_{12}S$ is invertible and
\begin{equation}
\label{operatorT}
R = (N_{21}+N_{22}S)(N_{11}+N_{12}S)\inv,
\end{equation} 
where 
\begin{equation}
\label{matrixN}
N = \begin{bmatrix}
N_{11} & N_{12} \\
N_{21} & N_{22}
\end{bmatrix} := M\inv
\end{equation}
is partitioned compatible with \eqref{matrixM}. 
\item \label{i:yTu} $y = R(u)$ if and only if $z = S(v)$.
\item \label{i:TsPhi} $R$ satisfies \eqref{incrementalaveragenonnegativity} with supply rate $s_\Phi$ if and only if $S$ satisfies \eqref{incrementalaveragenonnegativity} with supply rate $s_\Sigma$.
\end{enumerate}
\end{proposition}

\begin{proof}
We first prove \ref{i:MN}. By definition,
$$
\begin{bmatrix}
I \\ S
\end{bmatrix} = M\begin{bmatrix}
I \\ R
\end{bmatrix} (M_{11}+M_{12}R)\inv,
$$
and thus
\begin{equation}
\label{relationNMST}
N\begin{bmatrix}
I \\ S
\end{bmatrix} = \begin{bmatrix}
I \\ R
\end{bmatrix} (M_{11}+M_{12}R)\inv.
\end{equation}
The first block of \eqref{relationNMST} yields $N_{11}+N_{12}S = (M_{11}+M_{12}R)\inv$ so we conclude that $N_{11}+N_{12}S$ is invertible. Moreover, from the second block, we conclude that $R$ satisfies \eqref{operatorT}.

Next, we prove \ref{i:yTu}. Suppose that $y = R(u)$. Then $v = M_{11}u+M_{12}R(u)$ so that $(M_{11}+M_{12}R)\inv(v) = u$. Thus, $z = (M_{21}+M_{22}R)(u) = S(v)$. 

Conversely, suppose that $z = S(v)$. By definition of $v$ and $z$, and by \eqref{relationNMST} we have
$$
\begin{bmatrix} u \\ y \end{bmatrix} = N\begin{bmatrix}
v \\ z
\end{bmatrix} = \begin{bmatrix}
I \\ R
\end{bmatrix} (M_{11}+M_{12}R)\inv(v),
$$
and therefore $y = R(u)$. 

Finally, we note that \ref{i:TsPhi} follows directly from \ref{i:yTu} using the relation \eqref{factorPhi}. This proves the proposition.
\end{proof}

\begin{remark}
If $M$ is given by
\begin{equation}
\label{Mpassivity}
M = \frac{\sqrt{2}}{2} \begin{bmatrix}
I & I \\ I & -I
\end{bmatrix}
\end{equation}
it is often referred to as the \emph{scattering transform}, and in this case $S = (I-R)(I+R)\inv$ is called the \emph{scattering operator} of $R$. The relation between (incremental) passivity of $R$ and (incremental) $L_2$-gain of its scattering operator is classical, see e.g. \cite[Sec. 2.4]{vanderSchaft2017}.
\end{remark}

\subsection{Nonexpansive and monotone operators}
\label{s:preliminariesmonotone}
Note that $S$ satisfies \eqref{incrementalaveragenonnegativity} with supply $s_\Sigma$ if and only if 
\begin{equation}
\label{nonexpansiveL2}
\norm{S(u)-S(v)}_{\calY} \leq \norm{u-v}_{\calU}
\end{equation}
for all $u,v \in \calU$. Clearly, the inequality \eqref{nonexpansiveL2} can be considered for operators on general normed spaces, not just $L_2$ or $\ell_2$. In convex analysis and optimization, \eqref{nonexpansiveL2} is typically referred to as a \emph{nonexpansiveness} property, see e.g. \cite{Bauschke2011}. We recall the relevant definitions in what follows. 

\begin{definition}
\label{d:nonexpansive}
Let $\calX$ and $\calZ$ be Banach spaces. The operator $R : D \subseteq \calX \to \calZ$ is 
\begin{itemize}
\item \emph{Lipschitz continuous} if there exists a nonnegative constant $\ell \in \mathbb{R}$ such that 
$$
\norm{R(x)-R(y)}_\calZ \leq \ell \norm{x-y}_\calX
$$
for all $x,y \in D$.
\item \emph{Nonexpansive} if it is Lipschitz continuous with constant $\ell = 1$, i.e.,
$$
\norm{R(x)-R(y)}_\calZ \leq \norm{x-y}_\calX
$$
for all $x,y \in D$.
\item \emph{Contractive} if it is Lipschitz continuous with $\ell < 1$.
\end{itemize}
If $\calX = \calZ$ then $R$ is called
\begin{itemize}  
\item \emph{Firmly nonexpansive} if 
\begin{align*}
&\norm{R(x)-R(y)}^2_\calX + \norm{(I-R)(x)-(I-R)(y)}^2_\calX \\ \leq &\norm{x-y}^2_\calX
\end{align*}
for all $x,y \in D$. 
\end{itemize}
If, in addition, $\calX$ is a Hilbert space then $R$ is called 
\begin{itemize}
\item \emph{monotone} if 
$$
\inner{x-y}{R(x)-R(y)}_\calX \geq 0
$$
for all $x,y \in D$. 
\end{itemize}
\end{definition}

Nonexpansiveness and monotonicity\footnote{Note that we focus on the operator-theoretic notion of monotonicity. Other concepts of monotonicity arise in the context of state-space systems, which has recently received attention in data-based studies as well \cite{Kawano2020,Revay2021}.} are important concepts in convex optimization. Indeed, the subdifferential of a proper convex function is monotone. Under suitable conditions, this implies that gradient descent algorithms converge \cite{Ryu2022}, which can be shown by resorting to nonexpansiveness properties of the involved operator. It should be clear that monotonicity boils down to an incremental IQC with passive supply \eqref{supplypassivity} in the special case of $\calX = L_2$. Within monotone operator theory, the relation between monotonicity and nonexpansiveness of the ``scattering operator" is well-known, and this forms an analogy to the discussion in Section~\ref{s:preliminariesscattering}. In fact, if $I+R$ is invertible then by \cite[Prop. 23.9(ii)]{Bauschke2011}, $R$ is monotone if and only if its resolvent $(I+R)\inv$ is firmly nonexpansive. Moreover, by \cite[Prop. 4.2(i)-(iii)]{Bauschke2011}, $(I+R)\inv$ is firmly nonexpansive if and only if $(I-R)(I+R)\inv$ is nonexpansive.

\section{Problem formulation}
\label{s:problem}

Let $\mathbb{T}$, $\calU$ and $\calY$ be as in Section~\ref{s:preliminariesdissipativity}. Consider a data set consisting of a collection of $n$ pairs of trajectories:
\begin{equation}
\label{data}
(u_i,y_i) \in \calU\times \calY \text{ for } i = 1,2,\dots,n.
\end{equation}
At a high level, the problem considered in this paper is the following: find an operator $\hat{H}:\calU \to \calY$ that fits the data, i.e., approximately maps each $u_i$ to $y_i$, while ensuring that $\hat{H}$ satisfies a \emph{system-theoretic property} $\Pi$. 

Given their importance in nonlinear system analysis, we will focus on incremental system properties. This means that $\Pi$ is chosen as one of the properties\footnote{In Section~\ref{s:otherproperties}, we will also briefly discuss the case of (non-incremental) dissipativity, i.e., items \ref{i:avnonneg} and \ref{i:dis} of Definition~\ref{d:allsystemproperties}.} \ref{i:incavnonneg}, \ref{i:incdis} or \ref{i:causal} of Definition~\ref{d:allsystemproperties}. The proposed problem is thus to bias the identification task in such a way that the identified model satisfies an incremental IQC, or incremental dissipativity/causality properties.

One of the motivations of this problem is to incorporate prior knowledge of the physics of the plant in the identification process. In biophysical systems, such prior knowledge appears in the context of the Hodgkin-Huxley system \cite{Hodgkin1952}, describing the electrical characteristics of an excitable cell. All experimental data point towards the \emph{monotonicity} of the so-called potassium ion channel of this system. Nonetheless, it turns out that existing state-space models of this system \cite{Hodgkin1952} are \emph{not} monotone, a fact that we prove in Section~\ref{s:potassiumcurrent}. As we will see, the results of this paper are directly applicable to identify monotone systems.

Another philosophy behind the proposed problem is that we wish to estimate models that are directly useful for system analysis and control. Dissipativity is a cornerstone of the analysis of feedback systems \cite{Desoer1975} and plays an important role in the stability and optimality of model predictive controllers \cite{Diehl2011,Muller2015}. The considered problem is therefore well-aligned with the \emph{identification for control} movement \cite{Gevers1995,Lindqvist2001,Bombois2004,Gevers2005}, that recognized the need for \emph{goal-oriented} identification.   

We summarize the problem as follows.\\
\noindent 
\textbf{Problem}: Consider the data \eqref{data} and let $\Pi$ be a system-theoretic property. Compute an operator $\hat{H} : \calU \to \calY$ that
\begin{enumerate}[label=(\roman*)]
\item fits the data, i.e., ensures that the error
$$
\sum_{i=1}^n \norm{y_i - \hat{H}(u_i)}_{\calY}
$$
is small; and
\item satisfies the property $\Pi$. \label{i:satisfynonnegative}
\end{enumerate}

We will formalize the notion of ``small error" by introducing a suitable cost functional in Section~\ref{s:regularizedleastsquares}.

\section{Regularized least squares}
\label{s:regularizedleastsquares}

By Proposition~\ref{p:scattering}\ref{i:yTu} and \ref{i:TsPhi}, the problem of identifying an operator satisfying \eqref{incrementalaveragenonnegativity} is equivalent to identifying a nonexpansive operator using transformed data samples. As such, the problem described in Section~\ref{s:problem} boils down to devising an identification approach that can take into account nonexpansiveness and causality properties. 

We will provide such an approach by resorting to the theory of \emph{reproducing kernels}, see e.g. \cite{Aronszajn1950,Scholkopf2001,Micchelli2005}. This means that we seek an operator $\hat{H} \in \calH$, where $\calH$ is a so-called \emph{reproducing kernel Hilbert space} \emph{(RKHS)}. We refer to the Appendix for a detailed review of reproducing kernel Hilbert spaces of operators. Within the theory of RKHSs, a well-established identification approach \cite{Pillonetto2010,Pillonetto2011,
Dinuzzo2015,Bottegal2016,Bouvrie2017,
Ramaswamy2018,Blanken2020,Hamzi2021} is to select $\hat{H}$ as the solution to the \emph{regularized least squares problem}
\begin{equation}
\label{regLSpf}
\min_{H \in \calH} \left( \sum_{i=1}^n \norm{y_i - H(u_i)}^2_\calY + \gamma \norm{H}_\calH^2 \right).
\end{equation}
Regularized least squares (and more generally, kernel methods) play an important role in machine learning algorithms such as support vector machines \cite{Boser1992,Suykens1999,Scholkopf2001}. Note that the first term in \eqref{regLSpf} promotes a good fit of the data (i.e., a low \emph{empirical risk}). The second term promotes a ``low complexity" model, in the form of an operator with small norm. The parameter $\gamma > 0$ balances between these two terms. An additional benefit of regularization is well-posedness of the problem \eqref{regLSpf}: its solution exists and is unique for any RKHS $\calH$ and any $\gamma >0$, see Appendix~\ref{s:regLS}. 

Once we have agreed upon the identification approach \eqref{regLSpf}, the problem becomes how to choose $\calH$ and $\gamma$ so that the estimated operator $\hat{H}$ is causal and/or nonexpansive. The main contributions of this paper are to prove that such choices exist, and to provide concrete guidelines on how to select such spaces and parameters.

\section{Identifying nonexpansive operators}
\label{s:nonexpansive}

In this section we introduce an approach to identify nonexpansive operators using the theory of reproducing kernels. Throughout this section, we let $\calU$ and $\calY$ be real Hilbert spaces. We note that relevant special cases of this are the spaces $L_2$ and $\ell_2$ discussed in Section~\ref{s:preliminariesdissipativity}. In addition, let $\calH$ be a real reproducing kernel Hilbert space of operators from $\calU$ to $\calY$. We remind the reader that all relevant definitions and results about RKHSs can be found in the Appendix. Towards our goal of identifying nonexpansive operators, the following lemma is instrumental. It can be interpreted as the incremental version of Lemma~\ref{l:boundHu}\ref{l:boundHub} in Appendix~\ref{s:basicdefinitionsRKHS}.

\begin{lemma}
\label{l:boundHuHv}
Let $\calH$ be a reproducing kernel Hilbert space of operators from $\calU$ to $\calY$. Denote its kernel by $K\!:\!\calU \times \calU \!\to\! \calB(\calY)$. For all $u,v \in \calU$ and all $H \in \calH$ we have that 
\begin{equation}
\label{ineqHuHv}
\begin{aligned}
\!\!\!\!\!&\norm{H(u) - H(v)}_\calY \leq \\ \!\!\!\!\!&\norm{H}_\calH \norm{K(u,u)-K(u,v)-K(v,u)+K(v,v)}_{\calB(\calY)}^{\frac{1}{2}}.
\end{aligned}
\end{equation}
\end{lemma}

\vspace{2mm}

\begin{proof}
For any $u \in \calU$, define the linear operator $K_u : y \mapsto K(\cdot,u)y$ as in Lemma~\ref{l:boundHu}. We know from the proof of Lemma~\ref{l:boundHu} that $K_u \in \calB(\calY,\calH)$. Therefore, for any $u,v\in\calU$, 
\begin{equation*}
\begin{aligned}
\inner{H(u)-H(v)}{H(u)-H(v)}_\calY &= \\ \inner{H(u)-H(v)}{H(u)}_\calY - \inner{H(u)-H(v)}{H(v)}_\calY 
&= \\ \inner{H}{K_u(H(u)-H(v))}_\calH - \inner{H}{K_v(H(u)-H(v))}_\calH 
&= \\ \inner{H}{(K_u-K_v)(H(u)-H(v))}_\calH 
&\leq \\ \norm{H}_\calH \norm{K_u-K_v}_{\calB(\calY,\calH)} \norm{H(u)-H(v)}_\calY,
\end{aligned}
\end{equation*}
where we have used the reproducing property \eqref{repprop} for the second equality. This implies that
\begin{equation}
\label{ineqHuHv2}
\norm{H(u)-H(v)}_\calY \leq \norm{H}_\calH \norm{K_u-K_v}_{\calB(\calY,\calH)}.
\end{equation}
Moreover, using \eqref{repprop} again, we see that for all $y \in \calY$:
\begin{equation*}
\begin{aligned}
\norm{(K_u-K_v)y}_{\calH}^2 = \inner{(K_u-K_v)y}{(K_u-K_v)y}_\calH 
&= \\ \inner{(K_u-K_v)y}{K_u y}_\calH - \inner{(K_u-K_v)y}{K_v y}_\calH 
&= \\ \inner{y}{(K(u,u)-K(u,v))y}_\calY - \inner{y}{(K(v,u)-K(v,v))y}_\calY 
&= \\ \inner{y}{(K(u,u)-K(u,v) -K(v,u)+K(v,v))y}_\calY 
&\leq \\ \norm{y}^2_\calY \norm{K(u,u)-K(u,v) -K(v,u)+K(v,v)}_{\calB(\calY)}.
\end{aligned}
\end{equation*}
It follows that $\norm{K_u-K_v}_{\calB(\calY,\calH)}$ is less than or equal to
\begin{align*}
\norm{K(u,u)-K(u,v) -K(v,u)+K(v,v)}_{\calB(\calY)}^{\frac{1}{2}}.
\end{align*}
Combining this with \eqref{ineqHuHv2} yields \eqref{ineqHuHv}, as desired. 
\end{proof}

In view of \eqref{ineqHuHv}, it is tempting to impose the following condition on $K$:
\begin{equation}
\label{nonexpansivekernel}
\norm{K(u,u)-K(u,v)-K(v,u)+K(v,v)}_{\calB(\calY)}^{\frac{1}{2}} \leq \norm{u-v}_\calU.
\end{equation}
Indeed, if \eqref{nonexpansivekernel} holds then every $H\in \calH$ would satisfy $\norm{H(u) - H(v)}_\calY \leq \norm{H}_\calH \norm{u-v}_\calU$. That is, $H$ is Lipschitz continuous and nonexpansive if, in addition, $\norm{H}_\calH \leq 1$. We will call kernels that satisfy \eqref{nonexpansivekernel} \emph{nonexpansive}. Note that this notion of nonexpansiveness is not the same as the one in Definition~\ref{d:nonexpansive}, because it involves kernels, that is, mappings of \emph{two variables} on $\calU$. Nonetheless, the term ``nonexpansive" is natural also in the context of kernels because it is intimately related to Definition~\ref{d:nonexpansive} through \emph{feature maps}. Indeed, let $K:\calU\times\calU\to\calB(\calY)$ be symmetric positive semidefinite. Then, by Theorem~\ref{t:featuremap} in Appendix~\ref{s:featuremaps}, $K(u,v) = \phi(u)^* \phi(v)$ for some $\phi: \calU \to \calB(\calY,\calW)$, where $\calW$ is a Hilbert space. Next, we rewrite
\begin{equation*}
\begin{aligned}
K(u,u)-K(u,v)-K(v,u)+K(v,v)&
= \\ (\phi(u) - \phi(v))^*(\phi(u) - \phi(v))&.
\end{aligned}
\end{equation*}
By \cite[Fact 2.18(ii)]{Bauschke2011}, this means that
\begin{align*}
\norm{K(u,u)-K(u,v)-K(v,u)+K(v,v)}_{\calB(\calY)}^{\frac{1}{2}}& = \\ \norm{\phi(u) - \phi(v)}_{\calB(\calY,\calW)}&.
\end{align*}
Therefore, $K$ is nonexpansive if and only if all its associated feature maps are nonexpansive in the sense of Definition~\ref{d:nonexpansive}.

\begin{definition}
\label{d:nonexpansivekernel}
A mapping $K: \calU \times \calU \to \calB(\calY)$ is \emph{nonexpansive} if \eqref{nonexpansivekernel} holds for all $u,v \in \calU$.
\end{definition}

At this point, it is important to realize that the nonexpansiveness property is \emph{not} satisfied by every symmetric positive semidefinite kernel. As such, $K$ has to be designed judiciously so that \eqref{nonexpansivekernel} holds. To illustrate this further, we consider the following (non)example.

\begin{example}
\label{e:nonexpansivespline}
A class of kernels that has great success in linear system identification are the so-called \emph{stable spline kernels}, see e.g. \cite{Pillonetto2014}. In fact, these kernels have remarkable performance when applied to the problem of estimating impulse response functions \cite[Sec. 13.5]{Pillonetto2014}. However, as we will see next, these kernels do not possess the nonexpansive property and are therefore less suited for identification of nonexpansive (and thus, dissipative) systems. We focus on the first order stable spline kernel, defined by 
\begin{equation}
\label{stablespline}
K(u,v) = e^{-\beta \max\{u,v\}}, \:\: u,v \in \calU := \mathbb{R}^+,
\end{equation}
where $\beta > 0$, see \cite[Eq. (79)]{Pillonetto2014}. A simple calculation shows that
\begin{equation}
\label{stablesplinenonexpansive}
|K(u,u)+K(v,v)-2K(u,v)| = |e^{-\beta u} - e^{-\beta v}|.
\end{equation} 
Now, fix any $v \in \calU$ and suppose $u \geq v$. Define $x:= u-v$. Then \eqref{stablesplinenonexpansive} is equal to
$$
f(x) = e^{-\beta v} (1-e^{-\beta x}),
$$
so that $\frac{d}{dx}f(0) = \beta e^{-\beta v} > 0$. However, $\frac{d}{dx} x^2|_{x=0} = 0$. This means that for $x > 0$ sufficiently small, \eqref{nonexpansivekernel} is violated. We conclude that \eqref{stablespline} is not nonexansive for any $\beta >0$. 
\end{example}

Despite the fact that stable spline kernels do not satisfy Definition~\ref{d:nonexpansivekernel}, it turns out that several other kernels \emph{are}, in fact, nonexpansive. 

\begin{proposition}
\label{p:nonexpansivekernels}
The following (scalar-valued) symmetric positive semidefinite kernels $k: \calU \times \calU \to \mathbb{R}$ are nonexpansive:
\begin{itemize}
\item The bilinear kernel \eqref{bilinearkernel}.
\item The Gaussian kernel \eqref{Gaussiankernel} whenever $\sigma \geq \sqrt{2}$. 
\item The ``scaled" Laplacian kernel 
\begin{equation}
\label{scaledLaplaciankernel}
k(u,v) = (1+\norm{u-v}_\calU) e^{-\norm{u-v}_\calU}.
\end{equation}
\item The kernel
\begin{equation}
\label{fractionkernel}
k(u,v) = (c + \norm{u-v}_\calU^2)^{-d}
\end{equation} 
where $c \geq 0$ and $d > 0$ are reals satisfying $2d \leq c^{d+1}$. 
\end{itemize} 

In addition, $K: \calU \times \calU \to \calB(\calY)$ is both symmetric positive semidefinite and nonexpansive if
\begin{itemize}
\item $K(u,v) = k(u,v) R$ where $k : \calU \times \calU \to \mathbb{R}$ is a symmetric positive semidefinite and nonexpansive kernel, and $R \in \calB(\calY)$ is self-adjoint and positive with $\norm{R}_{\calB(\calY)} \leq 1$. 
\item $K(u,v) = \sum_{i = 1}^r \alpha_i K_i(u,v)$, where $K_i : \calU \times \calU \to \calB(\calY)$ is nonexpansive symmetric positive semidefinite, $\alpha_i \geq 0$ for all $i = 1,2\dots,r$ and $\sum_{i = 1}^r \alpha_i \leq 1$. 
\item $K(u,v) = R L(u,v) R^*$, where $L:\calU \times \calU \to \mathbb{R}$ is a symmetric positive semidefinite and nonexpansive kernel, and $R \in \calB(\calY)$ satisfies $\norm{R}_{\calB(\calY)} \leq 1$. 
\end{itemize}
\end{proposition}

\begin{proof}
It is clear that the bilinear kernel is nonexpansive, because for this kernel \eqref{nonexpansivekernel} holds with equality. To prove the statement for the Gaussian kernel, note that the squared left hand side of \eqref{nonexpansivekernel} equals
$$
2-2e^{-\frac{\norm{u-v}_\calU^2}{\sigma^2}}, 
$$
because the latter expression is nonnegative for all $u,v \in \calU$. By defining $x:= \norm{u-v}_\calU^2$, it thus suffices to prove that $f(x) := 2-2e^{-\frac{x}{\sigma^2}} \leq x$ for all $x \geq 0$. Note that $f(0) = 0$, thus this inequality holds at $0$. Furthermore, $\frac{df(x)}{dx} = \frac{2}{\sigma^2} e^{-\frac{x}{\sigma^2}}$, which is smaller or equal to $1$ for all $x \geq 0$ if $\sigma^2 \geq 2$. Therefore, we conclude that the Gaussian kernel is nonexpansive for $\sigma \geq \sqrt{2}$. 

Next, we note that both kernels \eqref{scaledLaplaciankernel} and \eqref{fractionkernel} are symmetric positive semidefinite by Proposition~\ref{p:scalarkernels}\ref{itemradial} because $(1+x)e^{-\sqrt{x}}$ and $(c+x)^{-d}$ are completely monotone functions. We will now prove that they are also nonexpansive, starting with the scaled Laplacian kernel. The squared left hand side of \eqref{nonexpansivekernel} equals
$$
2-2 (1+\norm{u-v}_\calU) e^{-\norm{u-v}_\calU},
$$
because the latter expression is nonnegative for all $u,v\in \calU$. Thus, by defining $x := \norm{u-v}_\calU$, it suffices to prove that $f(x):= 2-2(1+x)e^{-x} \leq x^2$ for all $x \geq 0$. Clearly, this inequality is satisfied for $x = 0$. In addition, we note that $\frac{df(x)}{dx} = 2x e^{-x} \leq 2x$ for all $x \geq 0$. As such, \eqref{scaledLaplaciankernel} is nonexpansive.
Next, we consider \eqref{fractionkernel}. For this kernel, the squared left hand side of \eqref{nonexpansivekernel} boils down to 
$$
2c^{-d} -2(c + \norm{u-v}_\calU^2)^{-d},
$$
because the latter is nonnegative for all $u,v\in\calU$. Therefore, similar as before, it suffices to show that $f(x) := 2c^{-d} -2(c + x)^{-d} \leq x$ for all $x \geq 0$. Clearly, this holds for $x = 0$. In addition, note that $\frac{df(x)}{dx} = 2d(c+x)^{-d-1}$, and thus, $\frac{df(0)}{dx} = 2d c^{-d-1} \leq 1$ by the hypothesis on $c$ and $d$. Moreover, since $\frac{d^2 f(x)}{dx^2} = -2d(d+1)(c+x)^{-d-2} \leq 0$ for all $x \geq 0$, we have that $\frac{df(x)}{dx} \leq 1$ for all $x \geq 0$ and thus $f(x) \leq x$ for all $x \geq 0$. We conclude that \eqref{fractionkernel} is nonexpansive if $c$ and $d$ satisfy $2d \leq c^{d+1}$. 

Finally, we prove the last three statements. Let $K(u,v) = k(u,v) R$ with $k$ a scalar-valued nonexpansive kernel and $R$ self-adjoint, positive and $\norm{R}_{\calB(\calY)} \leq 1$. By Proposition~\ref{p:separablekernels}, $K$ is symmetric positive semidefinite. In addition, 
\begin{align*}
\norm{K(u,u)-K(u,v)-K(v,u)+K(v,v)}_{\calB(\calY)} &\leq \\ |k(u,u)-k(u,v)-k(v,u)+k(v,v)| \norm{R}_{\calB(\calY)}
&\leq \\ \norm{u-v}_\calU^2&,
\end{align*}
that is, $K$ is nonexpansive. Next, let $K(u,v) = \sum_{i = 1}^r \alpha_i K_i(u,v)$, where each $K_i$ is nonexpansive, $\alpha_i \geq 0$ and $\sum_{i=1}^r \alpha_i \leq 1$. Clearly, $K$ is symmetric positive semidefinite by repeated application of the first item of Proposition~\ref{p:combiningkernels}. Moreover, 
\begin{align*}
\norm{K(u,u)-K(u,v)-K(v,u)+K(v,v)}_{\calB(\calY)} 
&\leq \\ \sum_{i=1}^r \alpha_i \norm{K_i(u,u)-K_i(u,v)-K_i(v,u)+K_i(v,v)}_{\calB(\calY)} 
&\leq \\ \sum_{i=1}^r \alpha_i \norm{u-v}_\calU^2 
&\leq \\ \norm{u-v}_\calU^2&,
\end{align*}
therefore $K$ is nonexpansive. 

Finally, consider $K(u,v) = R L(u,v) R^*$, where $L$ is nonexpansive and $\norm{R}_{\calB(\calY)} \leq 1$. In this case, 
\begin{align*}
\norm{K(u,u)-K(u,v)-K(v,u)+K(v,v)}_{\calB(\calY)} 
&\leq \\ \norm{R}_{\calB(\calY)}^2 \norm{L(u,u)-L(u,v)-L(v,u)+L(v,v)}_{\calB(\calY)} &\leq \\ \norm{u-v}_\calU^2&,
\end{align*}
again showing that $K$ is nonexpansive. 
\end{proof}

In the following theorem we summarize the main result of this section, which shows how nonexpansive kernels give rise to nonexpansive estimated operators. 

\begin{theorem}
\label{t:nonexpansive}
Let $\bar{u} := (u_1,u_2,\dots,u_n) \in \calU^n$ and $\bar{y} := (y_1,y_2,\dots,y_n) \in \calY^n$ be data. 
Consider a symmetric positive semidefinite kernel $K: \calU\times \calU \to \calB(\calY)$, and let $\calH$ be its associated reproducing kernel Hilbert space. Assume that $K$ is nonexpansive. Then the following statements hold:
\begin{enumerate}[label=(\alph*)]
\item Every $H \in \calH$ is Lipschitz continuous. \label{t:nonexpansive1}
\item The solution $\hat{H} \in \calH$ to the regularized least squares problem \eqref{regLSpf} has norm
$$
\norm{\hat{H}}_\calH = \norm{G\sq(G+\gamma I)\inv \bar{y}}_{\calY^n},
$$
where $G :\calY^n \to \calY^n$ is the Gram operator associated with $K$ and $\bar{u}$.
\label{t:nonexpansivenormHhat}
\item Thus, $\hat{H}$ is nonexpansive if $\gamma > 0$ satisfies
\begin{equation}
\label{conditiongamma}
\norm{G\sq(G+\gamma I)\inv \bar{y}}_{\calY^n} \leq 1.
\end{equation}
\label{t:nonexpansive2}
\end{enumerate}  
\end{theorem}

\begin{proof}
We first prove statement \ref{t:nonexpansive1}. By Lemma~\ref{l:boundHuHv}, we have that \eqref{ineqHuHv} holds for any $u,v \in \calU$ and $H \in \calH$. Moreover, because $K$ is nonexpansive, we conclude that 
$$
\norm{H(u) - H(v)}_\calY \leq \norm{H}_\calH \norm{u-v}_\calU,
$$
that is, $H$ is Lipschitz continuous.

Next, we prove \ref{t:nonexpansivenormHhat}. Let $\bar{c} := (c_1,c_2,\dots,c_n) \in \calY^n$ be the unique solution to \eqref{lineqregLS}, i.e., $\bar{c} = (G+\gamma I)\inv\bar{y}$.  Note that by \eqref{estimateregLS} and the reproducing property,
\begin{align*}
\inner{\hat{H}}{\hat{H}}_\calH &= \sum_{i = 1}^n \inner{\hat{H}}{K(\cdot,u_i)c_i}_\calH \\
&= \sum_{i = 1}^n \inner{c_i}{\hat{H}(u_i)}_\calY \\
&= \sum_{i = 1}^n \inner{c_i}{\sum_{j=1}^n K(u_i,u_j)c_j}_\calY \\
&= \inner{\bar{c}}{G\bar{c}}_{\calY^n} \\
&= \norm{G\sq(G+\gamma I)\inv \bar{y}}_{\calY^n}^2.
\end{align*}
As such, \ref{t:nonexpansivenormHhat} holds. Now, to prove \ref{t:nonexpansive2}, by \eqref{conditiongamma} the norm of $\hat{H}$ is less than or equal to $1$, which shows that $\hat{H}$ is nonexpansive. This proves the theorem. 
\end{proof}

\begin{remark}
Clearly, the condition \eqref{conditiongamma} can always be satisfied by selecting a sufficiently large $\gamma > 0$. Indeed,
\begin{align*}
\norm{G\sq(G+\gamma I)\inv \bar{y}}_{\calY^n} \leq \frac{1}{\gamma} \norm{G}_{\calB\left(\calY^n\right)}^{\frac{1}{2}} \norm{\bar{y}}_{\calY^n},
\end{align*}
where we have used \cite[Ch. 11, Thm. 12]{Bollobas1999} to obtain an upper bound on $\norm{(G+\gamma I)\inv}_{\calB(\calY^n)}$.
\end{remark}

By Theorem~\ref{t:nonexpansive} we take the standpoint that identifying a nonexpansive operator involves regularized least squares with i) a suitable (nonexpansive) kernel, and ii) a suitable regularization parameter $\gamma$. Of course, by \eqref{regLSpf} there is a clear trade-off between the data misfit and the norm of $\hat{H}$. As such, in practice it is often recommended to search for a $\gamma$ that ensures that \eqref{conditiongamma} holds (almost) with equality. This means that nonexpansiveness of $\hat{H}$ is guaranteed, but $\gamma$ is not chosen too large so that $\hat{H}$ still fits the data well. In some cases, it turns out to be beneficial to select $\gamma$ such that \eqref{conditiongamma} holds \emph{strictly}, ensuring that $\hat{H}$ is a contraction; we will discuss this in Section~\ref{s:nonexpansivetodissipative}. Further numerical aspects of the approach are discussed in Section~\ref{s:potassiumcurrent} where we treat numerical examples. 

\section{From nonexpansiveness to incremental IQC}
\label{s:nonexpansivetodissipative}

Returning to the problem discussed in Section~\ref{s:problem}, the status quo is that Theorem~\ref{t:nonexpansive} shows how to identify a nonexpansive operator $S$ in \eqref{operatorS}. The question that we address in this section is how to retrieve the original operator $R$ satisfying \eqref{incrementalaveragenonnegativity}, given $S$. In general, this problem is challenging because computing the inverse of a generic nonlinear operator is hard. Nonetheless, we will see that the properties on $S$ imposed by Theorem~\ref{t:nonexpansive} simplify matters and allow us to simulate $R$ effectively using fixed point algorithms. 

To facilitate the analysis, throughout this section we assume that $\calU$ and $\calY$ are spaces of functions as in Section~\ref{s:problem}. Suppose that the operator $S$ is obtained using the procedure of Theorem~\ref{t:nonexpansive}\ref{t:nonexpansive2}. We will additionally assume that $\gamma$ is selected so that \eqref{conditiongamma} holds strictly, meaning that $S$ is a \emph{contraction}. 

\begin{theorem}
\label{t:fixedpoint}
Suppose that $S:\calU \to \calY$ is a contraction with Lipschitz constant $\ell \in [0,1)$. Consider the matrix $M$ in \eqref{matrixM} and its inverse $N$ in \eqref{matrixN}. Assume that $N_{11}$ is invertible and 
\begin{equation}
\label{conditioneps}
\epsilon := \ell \norm{N_{11}\inv N_{12}}_{\calB(\calY,\calU)} < 1.
\end{equation} 
Then $N_{11}+N_{12}S$ is invertible so $R:\calU \to \calY$ in \eqref{operatorT} is well-defined. Moreover, for any $u^* \in \calU$, the output $y^* = R(u^*)$ can be computed as 
\begin{equation}
\label{computedoutputT}
y^* = (N_{21}+N_{22}S)(v^*).
\end{equation} 
Here $v^*$ is obtained via the Picard iteration $v^* = \lim_{k \to \infty} v^k$, where $v^0 \in \calU$ is arbitrary and
\begin{equation}
\label{picard}
v^{k+1} = N_{11}\inv u^* - N_{11}\inv N_{12} S(v^k) \text{ for } k \geq 0.
\end{equation} 
In addition, for any $k \geq 0$ we have that 
\begin{equation}
\label{convergencerate}
\norm{v^k - v^*}_\calU \leq \epsilon^k \norm{v^0 - v^*}_\calU.
\end{equation} 
\end{theorem}
\vspace{2mm}    

\begin{proof}
The proof is a consequence of the Banach fixed point theorem. Define $F:\calU\to\calU$ as 
$$
F: v \mapsto N_{11}\inv u^* - N_{11}\inv N_{12} S(v),
$$ 
for any given $u^* \in \calU$. Since $S$ is a contraction,
\begin{align*}
\norm{F(u)-F(v)}_{\calU} &= \norm{N_{11}\inv N_{12} S(u)-N_{11}\inv N_{12} S(v)}_\calU  \\
&\leq \ell \norm{N_{11}\inv N_{12}}_{\calB(\calY,\calU)} \norm{u-v}_\calU,
\end{align*} 
so $F$ is also a contraction by \eqref{conditioneps}. Therefore, by Banach fixed point theorem \cite[Thm. 1.48]{Bauschke2011}, the operator $F$ has a unique fixed point $v^* = \lim_{k \to \infty} v^k$, where $v^k$ is defined in \eqref{picard}. Also, by the same result, the rate of convergence is given in \eqref{convergencerate}. By definition of fixed point, $(N_{11}+N_{12}S)(v^*) = u^*$. Since $u^*$ was arbitrary and $v^*$ is unique for any $u^*$, this means that $N_{11}+N_{12}S$ is invertible. By definition of $R$ in \eqref{operatorT}, we see that $y^* = R(u^*)$ is given by \eqref{computedoutputT}. 
\end{proof}

We note that for the two common examples of $\Phi$ given in Section~\ref{s:problem}, the matrix $N$ is given by 
$$
\frac{\sqrt{2}}{2}
\begin{bmatrix}
I & I \\ I & -I
\end{bmatrix} \text{ and } \begin{bmatrix}
\frac{1}{\sqrt{\delta}} I & 0 \\ 0 & I
\end{bmatrix},
$$
respectively, which both satisfy the assumption \eqref{conditioneps}.

\section{Causality and incremental dissipativity}
\label{s:causalandincdis}
In this section, we focus on kernels for causality and incremental dissipativity. Throughout, we assume that $\calU$ and $\calY$ are the function spaces as in Section~\ref{s:preliminariesdissipativity}.

\subsection{Causality}

We first turn our attention to causality. We emphasize that (non-)causality has been investigated before in a linear setting \cite{Dinuzzo2015,Blanken2020}. Here we study causality for systems that are not necessarily linear. Recall that $P_T$ is defined in \eqref{definitionPT}. The operator $P_T$ is bounded on $\calU$ since $\norm{P_T u}_\calU \leq \norm{u}_\calU$. By linearity, $P_T$ is thus \emph{continuous}. Now, in order to learn causal operators from data, we define the notion of \emph{causal kernels}.

\begin{definition}
A mapping $K: \calU \times \calU \to \calB(\calY)$ is \emph{causal} if $K(\cdot,u)y :\calU \to \calY$ is causal for all $u\in \calU$ and $y\in\calY$. 
\end{definition}

Causal kernels are relevant because the RKHS of a causal, symmetric positive semidefinite kernel contains only causal operators. In particular, this means that the solution $\hat{H} \in \calH$ to the regularized least squares problem \eqref{regLSpf} is causal if $\calH$ is generated by a causal kernel. This is proven in the following theorem. 

\begin{theorem}
\label{t:causal}
Let $K:\calU\times\calU \to\calB(\calY)$ be a causal and symmetric positive semidefinite kernel. Let $\calH$ be its associated reproducing kernel Hilbert space. Then any $H \in \calH$ is causal. 
\end{theorem}

\begin{proof}
By linearity of $P_T$ is it clear that an operator of the form
$$
\sum_{i = 1}^n K(\cdot, u_i)y_i 
$$
is causal for any $u_i \in \calU$ and $y_i \in \calY$ ($i = 1,2,\dots,n$). In other words, the inner product space $\calH_1$, defined in Equation \eqref{H1} of the Appendix,
consists of causal operators. As discussed in Appendix~\ref{s:Moore-Aronszajn}, any operator $H\in\calH$ can be obtained as $H(u) = \lim_{i \to \infty} H_i(u)$ for all $u\in\calU$, where $H_1,H_2,\dots$ is a Cauchy sequence in $\calH_1$. This means that by continuity of $P_T$ and causality of $H_i$, we have that 
\begin{align*}
P_T(H(u)) &= P_T(\lim_{i \to \infty} H_i(u)) \\
&= \lim_{i \to \infty} P_T(H_i(u)) \\
&= \lim_{i \to \infty} P_T(H_i(P_T u)) \\
&= P_T(H(P_T u))
\end{align*}
for all $T \in \mathbb{T}$. That is, $H$ is causal. 
\end{proof}

Causal kernels can often be obtained by truncating the inputs of ``ordinary" kernels. We illustrate this in the special case of $\ell_2$ in the following example.

\begin{example}
Suppose that $\mathbb{T} = \{0,1,\dots,\tau\}$. Let $\calU = \ell_2(\mathbb{T},\mathbb{R}^m)$ and $\calY = \ell_2(\mathbb{T},\mathbb{R}^p)$. Let $K_t :\calU \times \calU \to \calB(\calY)$ be symmetric positive semidefinite for $t = 0,1,\dots,\tau$. We claim that the mapping $K:\calU\times \calU \to \calB(\calY)$, defined by 
$$
K(u,v)y = (K_0(P_0 u,P_0v)y(0),\dots,K_\tau(P_\tau u,P_\tau v)y(\tau)),
$$
is causal and symmetric positive semidefinite. We first show causality. Let $T \in \mathbb{T}$ and note that $P_T(K(u,v)y)$ is equal to
\begin{align*}
(K_0(P_0 u,P_0v)y(0),\dots,K_T(P_T u,P_T v)y(T),0,\dots,0) &= \\
P_T(K(P_T u,v)y)&
\end{align*}
for all $u \in \calU$ and $y\in \calY$. As such, $K$ is causal. Symmetry is a simple consequence of symmetry of $K_i$ ($i = 0,1,\dots,\tau$). Indeed, we see that
\begin{align*}
\inner{K(u,v)y}{z}_\calY &= \sum_{t=0}^\tau \inner{K_t(P_t u,P_t v)y(t)}{z(t)}_{\mathbb{R}^p} \\
&= \sum_{t=0}^\tau \inner{y(t)}{K_t(P_t v,P_t u) z(t)}_{\mathbb{R}^p} \\
&= \inner{y}{K(v,u)z}_\calY.
\end{align*}
Finally, to show positive semidefiniteness, let $n\in \mathbb{N}$, $u_i\in\calU$ and $y_i\in\calY$ for $i = 1,2,\dots,n$. Then the left hand side of \eqref{defpossemidef} can be rewritten as
\begin{align*}
\sum_{i = 1}^n \sum_{j=1}^n \sum_{t=0}^\tau \inner{y_i(t)}{K_t(P_t u_i,P_t u_j)y_j(t)}_{\mathbb{R}^p} &= \\
\sum_{t=0}^\tau \sum_{i = 1}^n \sum_{j=1}^n \inner{y_i(t)}{K_t(P_t u_i,P_t u_j)y_j(t)}_{\mathbb{R}^p} &\geq 0,
\end{align*}
making use of the fact that $K_t$ is positive semidefinite for all $t = 0,1,\dots,\tau$. We thus conclude that $K$ is causal and symmetric positive semidefinite. 
\end{example}

\subsection{Incremental dissipativity}
\label{s:incrementaldissipativity}

One approach to identify an incrementally dissipative operator is to select a kernel that imposes both causality and an  incremental IQC\footnote{This does mean that we restrict our attention to identifying \emph{causal} incrementally dissipative systems.}. Recall from Section~\ref{s:nonexpansivetodissipative} that we can impose incremental IQCs by identifying a contractive ``scattering" operator $S$. In fact, this ensures that $R$, given in \eqref{operatorT}, satisfies \eqref{incrementalaveragenonnegativity}. The question is now: under which conditions is $R$ \emph{causal} as well?

In general, the implication
\begin{equation}
\label{implicationcausality}
S \text{ causal and } N_{11}+N_{12}S \text{ invertible } \implies R \text{ causal}
\end{equation} 
does \emph{not} hold. In fact, in the special case that $M$ and $N$ are given by \eqref{Mpassivity}, the operator $R$ equals $R = 2(I+S)^{-1}-I$. Here, $(I+S)^{-1}$ can be regarded as the \emph{feedback system} obtained from operating $S$ in negative feedback with unit gain. The causality of such feedback systems is well-studied \cite{Willems1969,Saeks1970}. From this literature, counter examples to \eqref{implicationcausality} can be obtained, see for example \cite[Sec. 5.1]{Willems1969}.

Although \eqref{implicationcausality} thus does not hold in general, it turns out that this implication \emph{is} true if $S$ is a contraction. 

\begin{proposition}
\label{p:causalcontraction}
Consider the matrix $M$ in \eqref{matrixM} and its inverse $N$ in \eqref{matrixN}. Suppose that $S: \calU \to \calY$ is a causal contraction with constant $\ell \in [0,1)$, and assume that $N_{11}$ is invertible and \eqref{conditioneps} holds. Then the operator $R$ in \eqref{operatorT} exists and is causal. 
\end{proposition}

\begin{proof}
By Theorem~\ref{t:fixedpoint}, $R$ exists. Next, we prove causality. First, we note that $R$ is causal if and only if the implication 
\begin{equation}
\label{equivcausal}
P_T u = P_T v \implies P_T R(u) = P_T R(v)
\end{equation} 
holds for all $T \in \mathbb{T}$ and $u,v \in\calU$, see e.g. \cite{vanderSchaft2017}. 

Now, let $T \in \mathbb{T}$ and suppose that $u,v\in\calU$ are such that $P_T u = P_T v$. Define the outputs $y := (N_{11}+N_{12}S)^{-1}(u)$ and $z := (N_{11}+N_{12}S)^{-1}(v)$. Note that $u = (N_{11}+N_{12}S)(y)$ and $v = (N_{11}+N_{12}S)(z)$ so that
\begin{align*}
P_T N_{11} y + P_T N_{12} S(y) = P_T N_{11} z + P_T N_{12} S(z) 
\end{align*}
and 
\begin{align*}
P_T y -P_T z =  P_T N_{11}\inv N_{12}S(z) - P_T N_{11}\inv N_{12}S(y).
\end{align*}
This implies that $\norm{P_T y - P_T z}_\calU$ is upper bounded by
\begin{align*}
&\norm{N_{11}\inv N_{12}}_{\calB(\calY,\calU)} \norm{P_T S(y) - P_T S(z)}_\calY = \\
&\norm{N_{11}\inv N_{12}}_{\calB(\calY,\calU)} \norm{P_T S(P_T y) - P_T S(P_T z)}_\calY \leq \\
&\norm{N_{11}\inv N_{12}}_{\calB(\calY,\calU)} \norm{S(P_T y) - S(P_T z)}_\calY \leq \\ \ell &\norm{N_{11}\inv N_{12}}_{\calB(\calY,\calU)} \norm{P_T y - P_T z}_\calU.
\end{align*}
By \eqref{conditioneps}, we have $P_T y = P_T z$. As such, $(N_{11}+N_{12}S)^{-1}$ is causal by \eqref{equivcausal}. Because $S$ is causal, and the composition of two causal operators is again causal, this implies that $R$ in \eqref{operatorT} is causal. This proves the proposition. 
\end{proof}

By Proposition~\ref{p:causalcontraction}, the coup de gr\^{a}ce is that we can learn incrementally dissipative systems by identifying their causal contractive scatterings. We summarize this in the following result that follows from Theorems~\ref{t:nonexpansive} and \ref{t:causal} and Proposition~\ref{p:causalcontraction}.

\begin{theorem}
Consider a symmetric positive semidefinite kernel $K: \calU\times \calU \to \calB(\calY)$, and let $\calH$ be its associated reproducing kernel Hilbert space. Assume that $K$ is nonexpansive and causal. Then the following statements hold:
\begin{enumerate}[label=(\alph*)]
\item Every $H \in \calH$ is Lipschitz continuous and causal. \label{t:nonexpansivecausal1}
\item The solution $\hat{H} \in \calH$ to the regularized least squares problem \eqref{regLSpf} is contractive if $\gamma > 0$ satisfies \eqref{conditiongamma} with strict inequality.\label{t:nonexpansivecausal2}
\item Consider $\Phi$ satisfying Assumption \ref{a:Phi} and the matrix $M$ and its inverse $N$ in \eqref{matrixM} and \eqref{matrixN}. If \eqref{conditioneps} holds and \eqref{conditiongamma} is satisfied strictly, then the operator
$$
(N_{21}+N_{22}\hat{H})(N_{11}+N_{12}\hat{H})\inv
$$
is incrementally dissipative with respect to $s_\Phi$. 
\end{enumerate}  
\end{theorem}

\section{Illustrative example}
\label{s:potassiumcurrent}

We consider a model of the potassium ion channel, which is a component of the Hodgkin-Huxley dynamical system \cite{Hodgkin1952} describing the electrical characteristics of an excitable cell. This \emph{conductance-based} model relates the current through the potassium channel and the voltage across the cell's membrane through a nonlinear conductance. It is described by
\begin{equation}
\begin{aligned}
\label{potassiumconductance}
\dot{x} &= \alpha(u)(1-x) - \beta(u)x, \:\: x(0) = 0 \\
y &= g x^4 (u-\bar{u}),
\end{aligned}
\end{equation}
where $x: \mathbb{R} \to \mathbb{R}$ denotes the gating (state) variable, $u: \mathbb{R} \to \mathbb{R}$ denotes the (input) voltage across the cell's membrane, and $y : \mathbb{R} \to \mathbb{R}$ is the (output) potassium current per unit area. The functions $\alpha$ and $\beta$ depend on the voltage but not explicitly on time. Both $\bar{u} \in \mathbb{R}$ and $g \in \mathbb{R}$ are constants, called the \emph{reversal potential} and \emph{maximal conductance}, respectively. 

Owing to the original proposal by Hodgkin and Huxley, the constants $g$ and $\bar{u}$ are chosen as $g = 36$ and $\bar{u} = 12$, while the functions $\alpha$ and $\beta$ are given by
\begin{align*}
\alpha(u) &= 0.01 \frac{u+10}{e^{\frac{u+10}{10}}-1} \\
\beta(u) &= 0.125 e^{\frac{u}{80}}.
\end{align*}

Hodgkin and Huxley determined these parameters and functions on the basis of a series of step response experiments where the voltage $u$ was kept at different constant values, ranging from $-6mV$ to $-109mV$, and the corresponding current $y$ was measured, see \cite[Part II]{Hodgkin1952}. Using the model \eqref{potassiumconductance}, we have reconstructed the input-output data of \cite{Hodgkin1952} and the results are displayed in Figure~\ref{fig:constantinputs}. 

\begin{figure}[h!]
\centering
\includegraphics[width=0.5\textwidth]{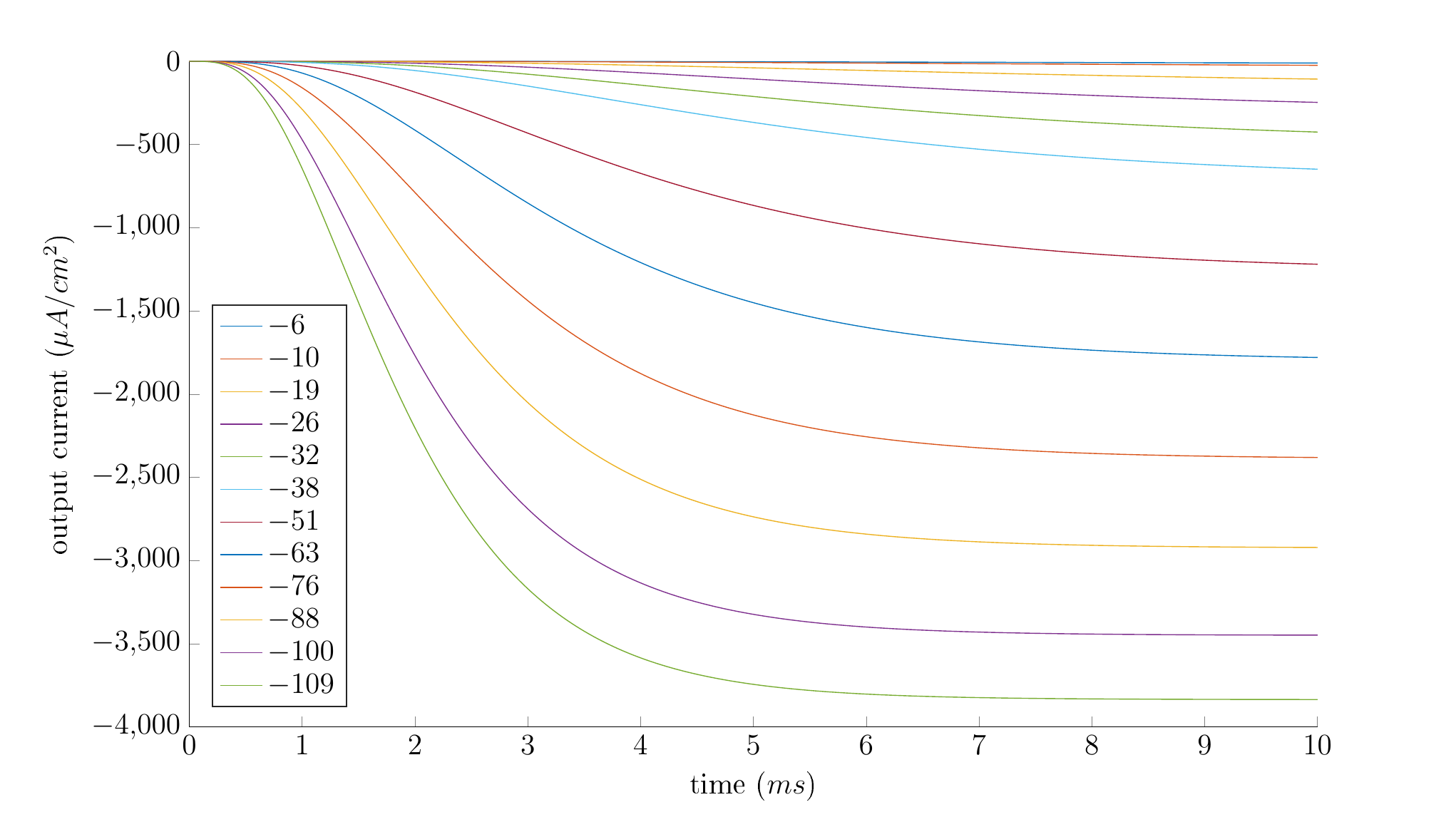}
\caption{Reconstruction of the experimental output currents measured by Hodgkin and Huxley \cite{Hodgkin1952} for various constant input voltages ranging from $-6 mV$ to $-109 mV$.}
\label{fig:constantinputs}
\end{figure}

These experimental data suggest that the input-output behavior of the potassium current defines a \emph{monotone} operator on $L_2([0,10],\mathbb{R})$, i.e. an operator satisfying the incremental IQC \eqref{incrementalaveragenonnegativity} with passive supply \eqref{supplypassivity}. This is the case because the output current $y_j$ corresponding to the constant input voltage $u_j$ satisfies $y_j(t) \leq y_k(t)$ for all $0\leq t \leq 10$ if $u_j \leq u_k$. 

It is interesting to note that the monotonicity property, suggested by all the experimental data, is \emph{not} carried over in Hodgkin and Huxley's model of the potassium current \eqref{potassiumconductance}. This fact can be demonstrated by injecting the two input voltages 
\begin{equation}
\label{V1V2}
\begin{aligned}
u_1(t) &= 25 \sin\left(\frac{12}{50} \pi t\right) \\
u_2(t) &= 25 \cos\left(\frac{14}{75} \pi t\right).
\end{aligned}
\end{equation} 

It can be readily verified that the inner product of $u_1-u_2$ and $y_1 - y_2$ is negative, more precisely, 
$$
\inner{u_1-u_2}{y_1-y_2}_{L_2} \approx -33.51.
$$
We therefore conclude that the model of the potassium current \eqref{potassiumconductance} is not monotone on $L_2([0,10],\mathbb{R})$.

Next, we will leverage the kernel-based approach of this paper to identify a model of the potassium current that \emph{is} monotone. For the sake of computation, we turn our attention to the sampled version of the problem. We take a sampling time of $0.5ms$, and study the evolution of the input and output of the potassium current over the time instants $t_0,t_1,\dots,t_{20}$, where $t_j = 0.5j$. We are thus interested in identifying an operator defined on $\ell_2(\{0,\dots,20\},\mathbb{R})$, mapping the voltage inputs $u(t_j)$ of the potassium channel to the current outputs $y(t_j)$ at the times $t_j$. To do this we will use the data in Figure~\ref{fig:constantinputs} at the time instants $t_j$ ($j=0,1,\dots,20$).

As an intermediate observation, we note that the sampled version of the model \eqref{potassiumconductance} is not monotone on $\ell_2$ either. This can once again be shown using the inputs \eqref{V1V2} restricted to the time instants $t_0,t_1,\dots,t_{20}$. 

We have observed better performance of kernel methods for normalized data. Therefore, we first preprocess the data by dividing all input samples by $a:= 978.7$ and all outputs by $b:=2.539\cdot 10^{4}$. Note that this will not affect our conclusions about monotonicity, for if $a,b >0$ then $R$ is a monotone operator if and only if
$$
\bar{R} : u \mapsto \frac{1}{b} R(au)
$$
is monotone\footnote{We will thus identify $\bar{R}$ and recover $R$ as $R(u) = b\bar{R}(\frac{1}{a}u)$.}. After scaling the data samples, we transform the input and output trajectories using the matrix $M$ in \eqref{Mpassivity} so that we are in the position to apply Theorem~\ref{t:nonexpansive}. We use a separable kernel of the form
$$
K(u,v) = (1+\norm{u-v}_{\ell_2}) e^{-\norm{u-v}_{\ell_2}} \cdot I,
$$
which is nonexpansive by Proposition~\ref{p:nonexpansivekernels}. We choose a value of $\gamma = 4.441\cdot 10^{-4}$ which implies that the left hand side of \eqref{conditiongamma} is $0.9903 < 1$. This results in an identified nonexpansive operator of the form \eqref{estimateregLS}. Lastly, we exploit the fixed point algorithm in Theorem~\ref{t:fixedpoint} to simulate the scattering operator of this identified system. The simulation results are reported in Figure~\ref{fig:identifiedoperator} for different constant input values. Note that the curves are obtained by interpolating between the output values at times $t_j$ ($j = 0,1,\dots,20$). For comparison, we also report the data samples using black crosses. 

We observe that the identified operator explains the data well, with a small misfit for larger input values and near-perfect reconstruction for smaller ones. Importantly, by Proposition~\ref{p:scattering} the identified operator is monotone.

\begin{figure}[h!]
\centering
\includegraphics[width=0.5\textwidth]{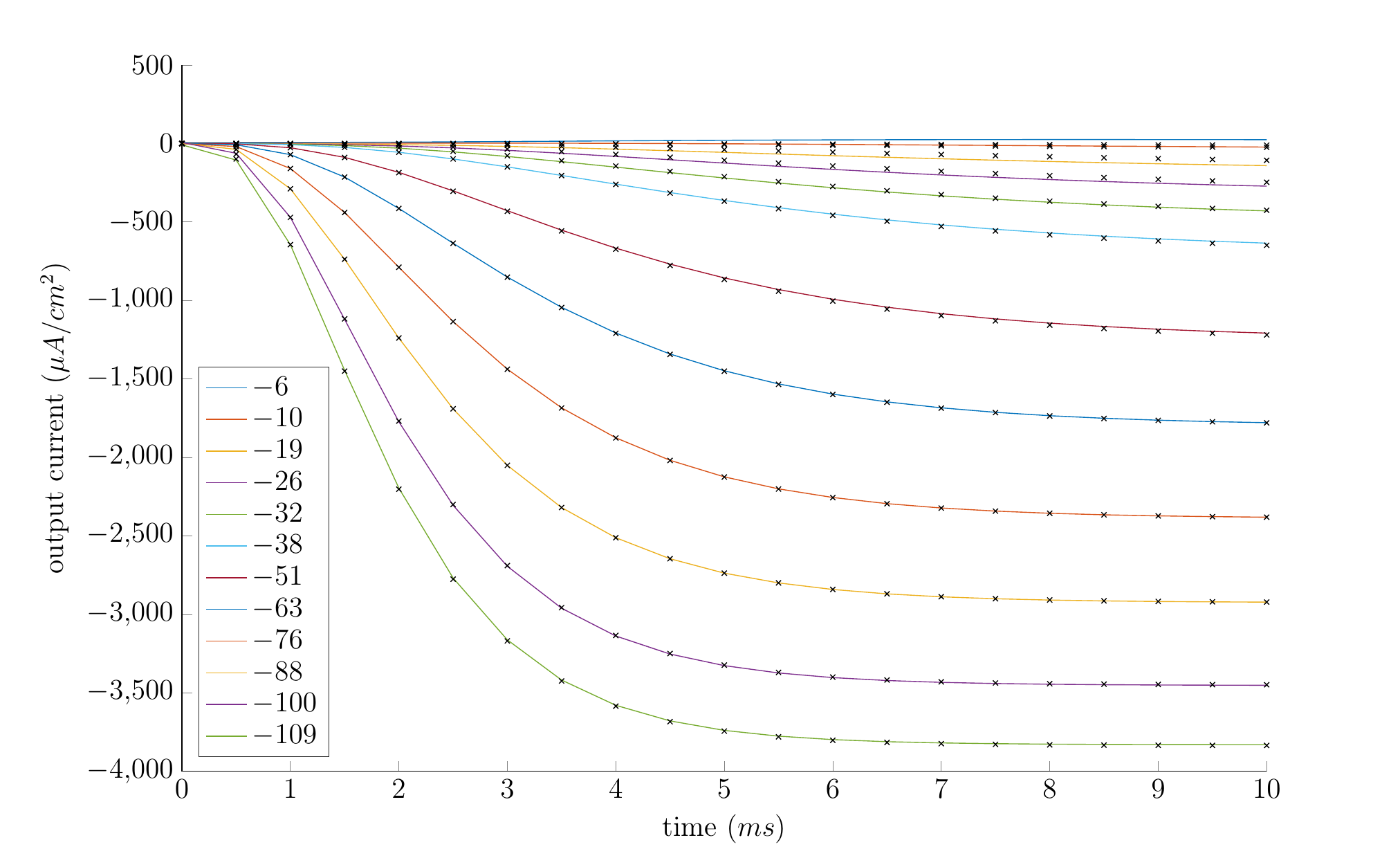}
\caption{Outputs of the identified operator for constant input voltages from $-6 mV$ to $-109 mV$ (in colors). The black crosses correspond to the data samples at times $t_0,t_1,\dots,t_{20}$.}
\label{fig:identifiedoperator}
\end{figure}

\section{Discussion}
\label{s:otherproperties}

\subsection{LTI systems versus trouble-making systems}

The kernel-based modeling framework 
studied in this paper is not necessarily advocated for linear system theory. 
The very success of  LTI system theory  is grounded in the availability of
a computational framework to derive state-space models with desired
input-output properties. Starting with the celebrated Kalman-Yakubovich-Popov
lemma, one of the most successful developments of control theory to date is the
LTI dissipativity theory of Willems \cite{Willems1972b}, that provides a solution to this very realization
problem via linear matrix inequalities. Dissipativity theory led to the foundation of algorithmic
control in the 1990s \cite{Boyd1994}. 

The philosophy of IQC theory is to leverage the LTI computational framework
to the analysis of interconnections that contain "trouble-making" systems in 
addition to the LTI models. Those trouble making systems  encompass everything that violates
the assumptions of finite-dimensional LTI models  (static nonlinearities, time-delays, rate and magnitude
actuator constraints, fading memory amplifiers,...). The characterization of their
input-output properties by IQCs is all what is needed to include them in the
analysis of the interconnected system.

The present paper can be regarded as an algorithmic method to impose incremental
integral quadratic constraints that encode a priori knowledge in a form that
is compatible with IQC system analysis. In that sense, the advocated framework
is input-output by nature and geared towards the decomposition of complex
systems as interconnections of simpler subsystems.

\subsection{Non-incremental system properties}

In the age of state-space modeling, the focus of nonlinear system analysis and design has drifted from incremental
to non-incremental system properties.  Nonlinear dissipativity theory \cite{Willems1972} has been flourishing, leading to system
designs with non-incremental properties such as passivity or finite gain. Classical nonlinear textbooks 
from the nineties (e.g. \cite{Nijmeijer1990,Sepulchre1997,
Isidori1985,Khalil2002}) are grounded in state-space modeling.
They emphasize the analysis  and design  of feedback systems with global but non-incremental properties,
with respect to a specific solution (usually a nominal  ``zero" equilibrium solution).

Data-driven modeling of (physical) dissipative state-space systems has proven to be difficult. In fact, even though established identification techniques such as prediction error and subspace methods have nonlinear extensions \cite{Ljung1978,Westwick1996,Ljung1999,Favoreel1999,Palanthandalam-Madapusi2005,Padoan2015}, these methods are not guaranteed to return dissipative models, even if the data-generating mechanism is known to be dissipative; see also \cite{Maciejowski1995} for similar problems involving stability. It turns out to be hard to take into account dissipativity in parametric identification methods, and the state-of-the-art is limited to linear time-invariant systems \cite{Rapisarda2011,Guta2016,Sivaranjan2019}.

The methods of this paper can also be used to identify systems satisfying (non-incremental) IQCs and dissipativity properties, i.e., items \ref{i:avnonneg} and \ref{i:dis} of Definition~\ref{d:allsystemproperties}.

Analogous to Proposition~\ref{p:scattering}, it can be shown that if $R$ is an operator for which $M_{11}+M_{12}R$ is invertible, then $R$ satisfies the IQC \ref{i:avnonneg} if and only if $S$ is \emph{bounded} with constant $1$, i.e.,
$$
\norm{S(u)}_\calY \leq \norm{u}_\calU
$$
for all $u \in \calU$. This means that   the story of Section~\ref{s:nonexpansive} can be mimicked to identify operators satisfying an IQC through their bounded scatterings. First, define the notion of bounded kernel as follows.

\begin{definition}
We say that $K:\calU\times\calU \to \calB(\calY)$ is \emph{bounded} if $\norm{K(u,u)}_{\calB(\calY)}^{\frac{1}{2}} \leq \norm{u}_\calU$ for all $u\in\calU$.
\end{definition}

Then, the following proposition shows how bounded operators can be identified. We omit the proof since it follows the same lines as the proof of Theorem~\ref{t:nonexpansive}, with the difference of invoking Lemma~\ref{l:boundHu}\ref{l:boundHub} instead of Lemma~\ref{l:boundHuHv}.

\begin{proposition}
\label{p:boundedoperator}
Let $\bar{u} := (u_1,u_2,\dots,u_n) \in \calU^n$ and $\bar{y} := (y_1,y_2,\dots,y_n) \in \calY^n$ be data. 
Consider a symmetric positive semidefinite kernel $K: \calU\times \calU \to \calB(\calY)$, and let $\calH$ be its associated reproducing kernel Hilbert space. Assume that $K$ is bounded. Then the following statements hold:
\begin{enumerate}[label=(\alph*)]
\item Every $H \in \calH$ is bounded.
\item If \eqref{conditiongamma} holds, then the solution $\hat{H} \in \calH$ to \eqref{regLSpf} is bounded with constant $1$.
\end{enumerate}  
\end{proposition}

A drawback of Proposition~\ref{p:boundedoperator} when compared to Theorem~\ref{t:nonexpansive} is that the ``scattering operator" \eqref{operatorT} of a bounded $\hat{H}$ may not exist (even if \eqref{conditiongamma} holds strictly). Furthermore, even if this scattering operator exists, it is not straightforward to simulate it in general, as fixed point algorithms are not directly applicable.  

These intricacies perhaps provide further evidence of the non-trivial gap between non-incremental and incremental properties in nonlinear system theory. In the present framework, the focus on incremental properties appears to be  essential to the algorithmic tractability of the approach. This observation is aligned with recent developments in the analysis and simulation of large-scale nonlinear circuits \cite{Chaffey2021}. 
 
\section{Conclusions}
\label{s:conclusions}

In this paper, we have introduced a modeling framework that combines the fitting requirements of data-based learning with the input-output requirements of system analysis. The framework is grounded on regularized least squares in reproducing kernel Hilbert spaces of operators. As our main results, we have shown how the kernel and regularization parameter can be selected in order to ensure incremental integral quadratic constraints, dissipativity and causality properties of the identified operator. These properties are central to  the analysis of feedback systems \cite{Desoer1975,vanderSchaft2017}, the design of model predictive controllers \cite{Diehl2011,Muller2015}, and the convergence of optimization algorithms \cite{Lessard2016}. 
The limitation of state-space modeling to achieve those objectives was illustrated with the simplest models of nonlinear circuits, in particular the celebrated model of the potassium current of  Hodgkin-Huxley model \cite{Hodgkin1952}.

The results of this paper are relevant beyond identification and control. In a machine learning context, it is widely recognized that (deep) neural networks have impressive generalization performance but lack robustness. Therefore, Lipschitz bounds are computed as a proxy for robustness \cite{Tsuzuku2018,Fazlyab2019}, and algorithms have been developed to train neural networks with given Lipschitz constants \cite{Revay2020,Pauli2022}. Related to this, Theorem~\ref{t:nonexpansive} of the current paper is directly applicable to learn operators with given Lipschitz constants on general Hilbert spaces.

There are many questions to explore beyond the proposal of this paper. Data-driven control has been popular in the recent years but so far primarily restricted to linear models. The present paper opens an avenue for the design of data-driven control design including ``trouble-making" blocks that violate linearity or invariance assumptions. Also, 
 we have focused on identifying operators satisfying integral quadratic constraints with a \emph{given} supply rate $\Phi$. It would be of interest to consider the simultaneous design of $\Phi$ and the identified operator $\hat{H}$. This problem would be related to the question of learning dissipativity properties from data \cite{Montenbruck2016,Romer2017}. The additional access to a model $\hat{H}$ would, however, open up new data-driven simulation and predictive control strategies. 

\appendix

In this appendix, we provide the background material on reproducing kernel Hilbert spaces that we need for the paper. Thus far, the systems and control community has primarily focused on RKHSs of functions with scalar outputs, see e.g. \cite{Pillonetto2010,Pillonetto2011,Pillonetto2014,
Dinuzzo2015,Bottegal2016,Bouvrie2017,
Ramaswamy2018,Blanken2020,Hamzi2021}. Since we study operators mapping into a (possibly infinite-dimensional) Hilbert space $\calY$, we will provide a sufficiently detailed review of the general theory of operator-valued reproducing kernels. For this background material we are indebted to Micchelli and Pontil \cite{Micchelli2004,
Micchelli2005}, Carmeli \emph{et al.} \cite{Carmeli2006}, Caponnetto \emph{et al.}
\cite{Caponnetto2008} and Kadri \emph{et al.} \cite{Kadri2016}. We will refer to several results in this literature, and give our own proofs whenever appropriate.

\subsection{Basic definitions and properties}
\label{s:basicdefinitionsRKHS}
Throughout, we let $\calU$ and $\calY$ be real Hilbert spaces\footnote{We note that all results of this section remain true if $\calU$ is merely a \emph{subset} of a Hilbert space. In this case, the notation $\inner{\cdot}{\cdot}_\calU$ and $\norm{\cdot}_\calU$ should be understood as the \emph{restriction} of the inner product and norm to $\calU$.}. In addition, we consider a Hilbert space $\calH$ of operators from the set of inputs $\calU$ to the output space $\calY$. 

\begin{definition}
\label{d:reproducingkernel}
A mapping $K:\calU\times\calU \to \calB(\calY)$ is called a \emph{reproducing kernel} for $\calH$ if the following two properties hold:
\begin{itemize}
\item $K(\cdot,u)y : \calU \to \calY$ is a member of $\calH$ for all $u\in\calU$ and $y \in \calY$;
\item The \emph{reproducing property} holds: for all $u\in\calU$, $y \in \calY$ and $H\in\calH$:
\begin{equation}
\label{repprop}
\inner{y}{H(u)}_\calY = \inner{H}{K(\cdot,u)y}_\calH.
\end{equation}
\end{itemize}
We say that $\calH$ is a \emph{reproducing kernel Hilbert space} if it admits a reproducing kernel. 
\end{definition}

Definition~\ref{d:reproducingkernel} forms the basis for our discussion. In the literature, other definitions of RKHSs have been given \cite{Micchelli2005}, and we point out the connection to Definition~\ref{d:reproducingkernel} in Theorem~\ref{t:contlinfunc}.

\begin{remark}
In the case that $\calH$ consists of \emph{scalar-valued} functions from $\calU$ to $\calY = \mathbb{R}$, also the reproducing kernel $K$ is real-valued. In this case, by using the standard inner product $\inner{a}{b}_\mathbb{R} = ab$, \eqref{repprop} is reduced to $H(u) = \inner{H}{K(\cdot,u)}_\calH$ for all $u \in \calU$ and $H \in \calH$, which is also considered, e.g. in the classical work by Aronszajn \cite{Aronszajn1950}.
\end{remark}

Next, we state the following basic lemma whose proof essentially follows from \cite[Prop. 2.1]{Micchelli2005}. We will also provide a proof since the lemma plays an important role in the developments of this paper. 

\begin{lemma}
\label{l:boundHu}
Let $K$ be a reproducing kernel for $\calH$. The following properties hold:
\begin{enumerate}[label=(\alph*)]
\item For any $u \in \calU$ the linear operator $K_u \in \calL(\calY,\calH)$, defined by
$$
K_u : y \mapsto K(\cdot,u)y,
$$
is bounded, i.e., $K_u \in \calB(\calY,\calH)$.
\label{l:boundHua}
\item For all $u\in \calU$ and $H \in \calH$: 
$$
\norm{H(u)}_\calY \leq \norm{H}_\calH \norm{K(u,u)}_{\calB(\calY)}^{\frac{1}{2}}.
$$
\label{l:boundHub}
\end{enumerate}
\end{lemma}

\begin{proof}
We first prove statement \ref{l:boundHua}. Note that for any $u \in \calU$ and $y \in \calY$, the reproducing property and Cauchy-Schwartz inequality imply
\begin{align*}
\norm{K_u y}_\calH^2 = \inner{K_u y}{K_u y}_\calH &= \inner{y}{K(u,u)y}_\calY \\ &\leq \norm{y}_\calY^2 \norm{K(u,u)}_{\calB(\calY)},
\end{align*}
and thus $K_u$ is bounded and $\norm{K_u}_{\calB(\calY,\calH)} \leq \norm{K(u,u)}_{\calB(\calY)}^{\frac{1}{2}}$. Now, to prove \ref{l:boundHub}, note that
\begin{align*}
\norm{H(u)}_\calY^2 &= \inner{H(u)}{H(u)}_\calY = \inner{H}{K(\cdot,u)H(u)}_\calH \\ &\leq \norm{H}_\calH \norm{K_u}_{\calB(\calY,\calH)} \norm{H(u)}_\calY,
\end{align*}
so that
$$
\norm{H(u)}_\calY \leq \norm{H}_\calH \norm{K_u}_{\calB(\calY,\calH)} \leq \norm{H}_\calH \norm{K(u,u)}_{\calB(\calY)}^{\frac{1}{2}},
$$
which proves the lemma.
\end{proof}

One of the interesting consequences of Lemma~\ref{l:boundHu}\ref{l:boundHub} is that convergence \emph{in norm} implies \emph{pointwise} convergence in reproducing kernel Hilbert spaces. 

Subsequently, we note that $\calH$ has a reprocing kernel if and only if a suitable ``evaluation" functional is continuous. This equivalence is well-known for scalar-valued kernels \cite{Aronszajn1950}. Theorem~\ref{t:contlinfunc} articulates this fact for operator-valued kernels. The theorem shows the equivalence of the definition of RKHSs in \cite[Def. 2.1]{Micchelli2005} and Definition~\ref{d:reproducingkernel}.

\begin{theorem}
\label{t:contlinfunc}
$\calH$ is a reproducing kernel Hilbert space if and only if for every $u\in \calU$ and $y\in\calY$ the linear functional 
$$
L_{u,y} : H \mapsto \inner{y}{H(u)}_\calY
$$ 
is continuous, equivalently, $L_{u,y} \in \calB(\calH,\mathbb{R})$.
\end{theorem}

\begin{proof}
For linear operators, the equivalence of boundedness and continuity is well-known. As such, we only prove that $\calH$ is a reproducing kernel Hilbert space if and only if $L_{u,y} \in \calB(\calH,\mathbb{R})$. The necessity is straightforward to prove. Indeed, suppose that $\calH$ is a reproducing kernel Hilbert space and let $K$ be a reproducing kernel. Let $u\in\calU$ and $y\in\calY$. Then by Lemma~\ref{l:boundHu}\ref{l:boundHub} and Cauchy-Schwartz, 
$$
|L_{u,y}(H)| \leq \norm{K(u,u)}_{\calB(\calY)}^{\frac{1}{2}} \norm{y}_\calY \norm{H}_\calH,
$$
that is, $L_{u,y} \in \calB(\calH,\mathbb{R})$. 

Next, we prove sufficiency. Let $u\in\calU$ and $y\in\calY$. Suppose that $L_{u,y} \in \calB(\calH,\mathbb{R})$. By Riesz representation theorem \cite[Fact 2.17]{Bauschke2011}, there exists an operator $K_{u,y} \in \calH$ such that 
\begin{equation}
\label{Riesz}
\inner{y}{H(u)}_\calY = \inner{H}{K_{u,y}}_\calH.
\end{equation}
Since $K_{u,y}$ is linear in $y$, we can define $K_u \in \calL(\calY,\calH)$ as $K_u y := K_{u,y}$. Define the mapping $K :\calU\times \calU \to \calL(\calY)$ as $K : (u,v) \mapsto (K_v (\cdot))u$. Note that $K(\cdot,u)y = K_{u,y} \in \calH$ for all $u \in \calU$ and $y \in \calY$. Clearly,
$$
\inner{K(\cdot,u)y}{K(\cdot,v)z}_\calH = \inner{K(\cdot,v)z}{K(\cdot,u)y}_\calH
$$
for all $u,v\in\calU$ and $y,z\in \calY$. By exploiting the property~\eqref{Riesz} on both sides, we obtain
\begin{equation}
\label{adjoint}
\inner{z}{K(v,u)y}_\calY = \inner{y}{K(u,v)z}_\calY.
\end{equation}
By \cite[p. 48]{Akhiezer2013}, this implies that $K(u,v) \in \calB(\calY)$. We have thus shown that $K$ is a reproducing kernel for $\calH$, so $\calH$ is a reproducing kernel Hilbert space. This proves the theorem. 
\end{proof}

Lastly, it is also possible to prove that every reproducing kernel Hilbert space has \emph{exactly one} reproducing kernel (cf. \cite[Thm. 1]{Kadri2016}).

\begin{proposition}
\label{p:unique}
If $\calH$ admits a reproducing kernel then it is unique. 
\end{proposition}

\subsection{Positive semidefinite kernels and Moore-Aronszajn}
\label{s:Moore-Aronszajn}

Next, we investigate symmetry and positive semidefiniteness of mappings $K: \calU\times \calU \to \calB(\calY)$. These properties turn out to completely characterize the class of reproducing kernels. 

\begin{definition}
A mapping $K: \calU\times \calU \to \calB(\calY)$ is called \begin{itemize}
\item \emph{symmetric} if $K(u,v)^* = K(v,u)$.
\item \emph{positive semidefinite} if for all $n \in \mathbb{N}$, $u_1,u_2,\dots,u_n \in \calU$ and $y_1,y_2,\dots,y_n \in \calY$ we have that
\begin{equation}
\label{defpossemidef}
\sum_{i=1}^n \sum_{j=1}^n \inner{y_i}{K(u_i,u_j)y_j}_\calY \geq 0.
\end{equation} 
\end{itemize}
\end{definition}

Symmetry and positive semidefiniteness can also be expressed in terms of the so-called Gram operator associated with $K$. To do this, we define 
$$
\calY^n := \underbrace{\calY \times \calY \times \cdots \times \calY}_{n \text{ times}}.
$$
Since $\calY$ is a Hilbert space, $\calY^n$ is a Hilbert space with inner product 
$$
\inner{(y_1,y_2,\dots,y_n)}{(z_1,z_2,\dots,z_n)}_{\calY^n} := \sum_{i = 1}^n \inner{y_i}{z_i}_\calY.
$$
For $u_1,u_2,\dots,u_n \in \calU$ the \emph{Gram operator} $G: \calY^n \to \calY^n$ is defined as 
$$
G(y_1,\dots,y_n) = \left(\sum_{j=1}^n K(u_1,u_j)y_j,\dots,\sum_{j=1}^n K(u_n,u_j)y_j\right).
$$

We then have the following lemma.

\begin{lemma}
\label{l:spsd}
The mapping $K: \calU\times\calU \to \calB(\calY)$ is 
\begin{enumerate}[label=(\alph*)]
\item symmetric if and only if $G$ is self-adjoint for all $n \in \mathbb{N}$ and $u_1,u_2,\dots,u_n \in \calU$.
\label{l:s}
\item positive semidefinite if and only if $G$ is a positive operator for all $n \in \mathbb{N}$ and $u_1,u_2,\dots,u_n \in \calU$.
\label{l:psd}  
\end{enumerate}
\end{lemma}

\begin{proof}
First, we note that statement \ref{l:psd} follows directly from the definition of the Gram operator and the inner product on $\calY^n$. 

Thus, we focus on proving \ref{l:s}. Suppose that $G$ is self-adjoint for all $n \in \mathbb{N}$ and all $u_1,u_2,\dots,u_n \in \calU$. In particular, for $n = 2$ we can derive the following. Let $u_1,u_2 \in \calU$ be arbitrary. For any $y,z \in \calY$, we have that
\begin{align*}
G(0,y) &= (K(u_1,u_2)y,K(u_2,u_2)y), \: \text{ and} \\
G(z,0) &= (K(u_1,u_1)z,K(u_2,u_1)z).
\end{align*}
In addition, as $G = G^*$ we have
$$
\inner{(z,0)}{G(0,y)}_{\calY^2} = \inner{G(z,0)}{(0,y)}_{\calY^2}
$$
for all $y,z \in \calY$. This implies that
$$
\inner{z}{K(u_1,u_2)y}_\calY = \inner{y}{K(u_2,u_1)z}_\calY.
$$
That is, $K(u_1,u_2)^* = K(u_2,u_1)$ thus $K$ is symmetric. 

Conversely, let $K$ be symmetric. Let $u_1,u_2,\dots,u_n \in \calU$ and $(y_1,y_2,\dots,y_n),(z_1,z_2,\dots,z_n)\in\calY^n$, and note that
\begin{align*}
\inner{(z_1,z_2,\dots,z_n)}{G(y_1,y_2,\dots,y_n)}_{\calY^n} &= \\
\sum_{i = 1}^n \inner{z_i}{\sum_{j=1}^n K(u_i,u_j)y_j}_\calY 
&= \\ \sum_{i=1}^n \sum_{j=1}^n \inner{z_i}{K(u_i,u_j)y_j}_\calY &= \\ 
\sum_{i=1}^n \sum_{j=1}^n \inner{K(u_j,u_i)z_i}{y_j}_\calY &= \\
\sum_{j=1}^n \inner{\sum_{i=1}^n K(u_j,u_i)z_i}{y_j}_\calY &= \\
 \inner{G(z_1,z_2,\dots,z_n)}{(y_1,y_2,\dots,y_n)}_{\calY^n}.
\end{align*}
Therefore, $G = G^*$ which proves the lemma. 
\end{proof}

\begin{remark}
Clearly, if $K:\calU\times \calU \to \mathbb{R}$ is \emph{scalar-valued}, positive semidefiniteness of the kernel is equivalent to that of the \emph{Gram matrix}
\begin{equation}
\label{Gram}
\begin{bmatrix}
K(u_1,u_1) & K(u_1,u_2) & \cdots & K(u_1,u_n) \\ 
K(u_2,u_1) & K(u_2,u_2) & \cdots & K(u_2,u_n) \\ 
\vdots & \vdots & \ddots & \vdots \\
K(u_n,u_1) & K(u_n,u_2) & \cdots & K(u_n,u_n)
\end{bmatrix} 
\end{equation}
for all $n \in \mathbb{N}$ and $u_1,u_2,\dots,u_n \in \calU$. 
\end{remark}

The following elementary result, that follows directly from the reproducing property, highlights that reproducing kernels are both symmetric and positive semidefinite.  

\begin{proposition}
\label{p:spsdkernel}
If $K: \calU\times \calU \to \calB(\calY)$ is the reproducing kernel of some reproducing kernel Hilbert space, then it is symmetric and positive semidefinite. 
\end{proposition}

A far more profound result is the converse of Proposition~\ref{p:spsdkernel}: every symmetric positive semidefinite $K$ is the reproducing kernel of exactly one RKHS. This result, which is by no means trivial, was first proven in the case of scalar-valued kernels by Aronszajn \cite{Aronszajn1950}. He attributes the theorem ``essentially" to Moore, and the result is thus known as the Moore-Aronszajn theorem. For the operator-valued version, we refer to \cite{Micchelli2005}. We summarize the result as follows.

\begin{theorem}
A mapping $K:\calU\times\calU\to\calB(\calY)$ is the reproducing kernel for some reproducing kernel Hilbert space if and only if it is symmetric positive semidefinite. 

Moreover, if $K$ is symmetric positive semidefinite, then there exists a \emph{unique} reproducing kernel Hilbert space $\calH$ that admits $K$ as a reproducing kernel. 
\end{theorem}

The Moore-Aronszajn theorem is important because it provides a complete classification of reproducing kernel Hilbert spaces: every symmetric positive semidefinite $K$ defines a unique reproducing kernel Hilbert space and vice versa.

It also sheds light on the structure of the Hilbert space associated with a reproducing kernel. In fact, if $K$ is symmetric positive semidefinite, then 
\begin{equation}
\label{H1}
\calH_1 = \left\{ \sum_{i = 1}^n K(\cdot, u_i)y_i \mid n \in \mathbb{N}, u_i\in\calU, y_i\in\calY \right\}
\end{equation} 
can be shown to be an inner product space with inner product\footnote{Note that without loss of generality we can choose an equal number of $n$ terms in both functions (with the same $u_i\in\calU$) because some of the $y_i$'s or $z_i$'s can be selected to be zero.}
$$
\inner{\sum_{i = 1}^n K(\cdot, u_i)y_i}{\sum_{i = 1}^n K(\cdot, u_i)z_i\!\!}_{\calH_1} \!\!\!\!\!\!\!:=\! \sum_{i=1}^n\sum_{j=1}^n \inner{y_i}{K(u_i,u_j)z_j}_\calY
$$
and induced norm $\norm{\cdot}_{\calH_1} := \sqrt{\inner{\cdot}{\cdot}_{\calH_1}}$. Then $\calH$ is the completion of $\calH_1$ in the following sense: every $H\in\calH$ is the pointwise limit of a Cauchy sequence in $\calH_1$, i.e., $H(u) = \lim_{i \to \infty} H_i(u)$ for all $u\in\calU$, where $H_1,H_2,\dots$ is a Cauchy sequence in $\calH_1$. For a proof we refer to \cite{Aronszajn1950} and the accessible technical report \cite{Kalnishkan2009}. As pointed out in \cite{Micchelli2005}, these proofs are for scalar kernels but also extend to the operator-valued case considered here. 

\subsection{Feature maps}
\label{s:featuremaps}

Let $\calW$ be a Hilbert space and consider a mapping
$$
\phi : \calU \to \calB(\calY,\calW).
$$
We are interested in exploiting the mapping $\phi$ to define kernels $K : \calU \times \calU \to \calB(\calY)$ of the form
$$
K(u,v) = \phi(u)^* \phi(v).
$$
The mapping $\phi$ is often referred to as a \emph{feature map}. As also pointed out in \cite{Micchelli2005}, there is a close relationship between symmetric positive semidefinite kernels and their associated feature maps. 

\begin{theorem}
\label{t:featuremap}
Let $K : \calU \times \calU \to \calB(\calY)$. Then $K$ is symmetric positive semidefinite if and only if there exists a Hilbert space $\calW$ and a feature map $\phi: \calU \to \calB(\calY,\calW)$ such that for all $u,v \in \calU$:
\begin{equation}
\label{featurerep}
K(u,v) = \phi(u)^*\phi(v).
\end{equation}
\end{theorem}

\begin{proof}
We first prove sufficiency. Thus, we assume that a feature map exists and want to prove that $K$ in \eqref{featurerep} is symmetric positive semidefinite. Note that $K(u,v)^* = (\phi(u)^* \phi(v))^* = \phi(v)^*\phi(u) = K(v,u)$, thus $K$ is symmetric. In addition, 
$$
\inner{y}{K(u,v)z}_\calY = \inner{\phi(u)y}{\phi(v) z}_\calW.
$$
Hence for all $n \in \mathbb{N}$ and $(u_i,y_i) \in \calU \times \calY$ ($i = 1,2,\dots,n$) we have 
\begin{align*}
\sum_{i=1}^n\sum_{j=1}^n \inner{y_i}{K(u_i,u_j)y_j}_\calY &= \sum_{i=1}^n\sum_{j=1}^n \inner{\phi(u_i)y_i}{\phi(u_j) y_j}_\calW \\
&= \inner{\sum_{i=1}^n \phi(u_i)y_i}{\sum_{j=1}^n \phi(u_j)y_j}_\calW,
\end{align*}
and thus $K$ is positive semidefinite. 

To prove necessity, suppose that $K$ is symmetric positive semidefinite. Define the Hilbert space $\calW := \calH$ and the mapping $\phi : \calU \to \calB(\calY,\calH)$ as $\phi : v \mapsto K(\cdot,v)$. By the reproducing property \eqref{repprop} we have that
$$
\inner{H}{\phi(v)y}_\calH = \inner{H}{K(\cdot,v)y}_\calH = \inner{H(v)}{y}_\calY.
$$
This implies that the adjoint of $\phi(v)$ is just the evaluation operator, that is, $\phi(v)^* \in \calB(\calH,\calY)$ is given by $\phi(v)^* H = H(v)$. This means that $\phi(u)^* \phi(v)y = \phi(u)^* K(\cdot,v)y = K(u,v)y$ and therefore we have shown that $K(u,v) = \phi(u)^* \phi(v)$. We have thus shown the existence of a Hilbert space $\calW$ and mapping $\phi$ such that \eqref{featurerep} holds. This proves the theorem.
\end{proof}

\subsection{Regularized least squares}
\label{s:regLS}

An attractive feature of RKHSs is that several function estimation problems have an elegant and tractable solution if the underlying space has a reproducing kernel. In this section we focus on the regularized least squares problem
\begin{equation}
\label{regLS}
\min_{H \in \calH} \left( \sum_{i=1}^n \norm{y_i - H(u_i)}^2_\calY + \gamma \norm{H}_\calH^2 \right),
\end{equation}
where $y_i \in \calY$, $u_i \in \calU$ for $i = 1,2,\dots,n$ and $\gamma > 0$ is a scalar. It turns out that the solution to \eqref{regLS} is unique for any RKHS, see \cite[Thm. 4.1]{Micchelli2005}. 

\begin{theorem}[RegLS representer theorem]
\label{t:regLS}
Suppose that $\calH$ is a reproducing kernel Hilbert space of operators from $\calU$ to $\calY$ and let $K : \calU \times \calU \to \calB(\calY)$ be its reproducing kernel. There exists a unique solution $\hat{H}$ to \eqref{regLS}, which is given by 
\begin{equation}
\label{estimateregLS}
\hat{H}(\cdot) = \sum_{j=1}^n K(\cdot,u_j) c_j,
\end{equation}
where the coefficients $c_j \in \calY$ ($j = 1,2,\dots,n$) are the unique solution to the system of linear equations
\begin{equation}
\label{lineqregLS}
\sum_{j=1}^n (K(u_i,u_j) + \gamma \delta_{i,j}) c_j = y_i, \:\:i = 1,2,\dots,n,
\end{equation}
with $\delta_{i,j} \in \calB(\calY)$ equal to the identity operator if $i = j$ and the zero operator if $i\neq j$. 
\end{theorem}

We remark that the system of linear equations  \eqref{lineqregLS} can be written compactly using the Gram operator. Indeed, \eqref{lineqregLS} is equivalent to 
$$
(G+\gamma I)(c_1,\dots,c_n) = (y_1,\dots,y_n),
$$
where $G: \calY^n \to \calY^n$ is the Gram operator associated with $K$ and $u_1,u_2,\dots,u_n$.

\subsection{How to construct kernels?} 

We will conclude this section on background material by giving some concrete guidelines on how to construct symmetric positive semidefinite kernels. 
First off, as noted in \cite[Thm. 3]{Kadri2016}, it is possible to construct new kernels, given two symmetric positive semidefinite ones.  
\begin{proposition}
\label{p:combiningkernels}
Let $L,M : \calU \times \calU \to \calB(\calY)$ be symmetric positive semidefinite. 
\begin{itemize}
\item $K:= L+M$ is a symmetric positive semidefinite kernel. 
\item If $L(u,v)M(u,v) = M(u,v)L(u,v)$ $\forall u,v \in \calU$ then $K:=LM$ is symmetric positive semidefinite. 
\item $K := RLR^*$ is a symmetric positive semidefinite kernel for any $R \in \calB(\calY)$.
\end{itemize}
\end{proposition}

In addition, we can construct operator-valued kernels from scalar-valued ones \cite{Micchelli2004,Micchelli2005}.

\begin{proposition}
\label{p:separablekernels}
Let $k : \calU \times \calU \to \mathbb{R}$ be a symmetric positive semidefinite kernel and consider a self-adjoint positive operator $R \in \calB(\calY)$. Then $K: \calU \times \calU \to \calB(\calY)$, defined by $K(u,v) := k(u,v) R$, is symmetric positive semidefinite.
\end{proposition}

Kernels of the form described in Proposition~\ref{p:separablekernels} are often called \emph{separable} \cite{Alvarez2012,Kadri2016}. One of their advantages is fast computation of the solution to linear equations of the form \eqref{lineqregLS}. This is because the Gram operator of a separable kernel can be decomposed as a tensor product. Indeed, if we define the tensor product $A \otimes B : \calX^n \to \calY^m$ of the linear operators
\begin{align*}
&A: \!\mathbb{R}^n \!\to\! \mathbb{R}^m, A: (x_1,\dots,x_n) \mapsto (\sum_{j = 1}^n a_{1j}x_j,\dots,\sum_{j = 1}^n a_{mj}x_j),
\end{align*}
and $B: \calX \to \calY$ ($\calX$ and $\calY$ are Hilbert spaces) as
$$
A \otimes B : (z_1,\dots,z_n) \mapsto (\sum_{j = 1}^n a_{1j} B z_j,\dots,\sum_{j = 1}^n a_{mj} B z_j),
$$
then the Gram operator can be decomposed as
$$
G = \begin{bmatrix}
k(u_1,u_1) & \cdots & k(u_1,u_n) \\  
\vdots & \ddots & \vdots \\
k(u_n,u_1) & \cdots & k(u_n,u_n)
\end{bmatrix} \otimes R.
$$

In the next proposition we discuss some concrete examples of scalar-valued kernels. For this, we need the definition of \emph{complete monotonicity}.

\begin{definition}
A function $f:[0,\infty) \to \mathbb{R}$ is called \emph{completely monotone} if it is continuous on $[0,\infty)$, infinitely differentiable on $(0,\infty)$ and satisfies
$$
(-1)^n \frac{d^n}{dx^n} f(x) \geq 0
$$
for all $n \in \mathbb{N}$ and $x \in (0,\infty)$. 
\end{definition}

\begin{proposition}
\label{p:scalarkernels}
The following scalar-valued $k: \calU \times \calU \to \mathbb{R}$ are symmetric positive semidefinite.
\begin{enumerate}[label=(\alph*)]
\item The \emph{polynomial kernel}: 
$$
k(u,v) = (c + \inner{u}{v}_\calU)^d,
$$
with $c \geq 0$ and $d \in \mathbb{N}$.
\item As a special case, the \emph{bilinear kernel}: 
\begin{equation}
\label{bilinearkernel}
k(u,v) = \inner{u}{v}_\calU.
\end{equation} 
\item \label{itemradial} The \emph{radial basis function kernel}:
\begin{equation}
\label{radialkernel}
k(u,v) = f(\norm{u-v}_\calU^2),
\end{equation} 
where $f: [0,\infty) \to \mathbb{R}$ is completely monotone. 
\item In particular, the \emph{Gaussian kernel}: 
\begin{equation} 
\label{Gaussiankernel}
k(u,v) = e^{-\frac{\norm{u-v}^2_\calU}{\sigma^2}},
\end{equation} 
where $\sigma > 0$. 
\item And the \emph{Laplacian kernel}: 
\begin{equation}
\label{Laplaciankernel}
k(u,v) = e^{-\frac{\norm{u-v}_\calU}{\sigma}},
\end{equation} 
where $\sigma > 0$. 
\end{enumerate}
\end{proposition}

The fact that polynomial kernels are symmetric positive semidefinite is well-known, and can be found e.g. in the textbook \cite[Sec. 2.3]{Scholkopf2001}. The result \ref{itemradial} is due to Schoenberg \cite[Thm. 3]{Schoenberg1938}. Note that the positive semidefiniteness of Gaussian and Laplacian kernels follows from Schoenberg's result, noting that $e^{-\frac{x}{\sigma^2}}$ and $e^{-\frac{\sqrt{x}}{\sigma}}$ are completely monotone functions in $x$.

\bibliographystyle{IEEEtran}
\bibliography{references}

%\begin{IEEEbiography}[{\includegraphics[width=1in,height=1.25in,clip,keepaspectratio]{a1.png}}]{First A. Author} (M'76--SM'81--F'87)
%\end{IEEEbiography}

\end{document}